\documentclass[a4paper]{article}

\usepackage{amsmath} 

\usepackage{tikz}
\usepackage{amssymb}
\usepackage{verbatim} 
\usetikzlibrary{matrix}
\usepackage{color} 
\usepackage{stmaryrd}
\usepackage{amsfonts} 
\usepackage[all,cmtip]{xy}
\usepackage{catex} 
\usepackage{fancybox}
\usepackage{multicol}

\newtheorem{teor}{teor}
\newtheorem{cor}[teor]{Corollary}
\newtheorem{prop}{Proposition}
\newtheorem{lemma}{Lemma}
\newtheorem{proof}{Proof}

\newtheorem{defi}{Definition}

\newtheorem{oss}{Remark}

\newcommand{\fr}{\rightarrow}
\newcommand{\sig}{\mathcal S}
\newcommand{\C}{\mathcal C}
\newcommand{\R}{\mathcal R}
\newcommand{\D}{\phi}
\newcommand{\M}{\mathfrak M}
\newcommand{\Sim}{\mathfrak F}
\newcommand{\N}{\mathbb N}
\newcommand{\pro}{\; \square \;}

\newcommand{\sym}{\tau}

\newcommand{\e}{\mathit e}
\newcommand{\id}{\textbf{id}}

\newcommand{\comm}{\bigodot}

\newcommand{\rew}[2]{\xymatrix@C-1pc{{#1}\ar@3[r]&{#2}} }

\def\MLine#1{\par\hspace*{-\leftmargin}\parbox{\textwidth}{\[#1\]}}

\title{A Constructive Proof of Coherence for Symmetric Monoidal Categories Using Rewriting }
\author{Matteo Acclavio}
\date{}

\newcommand{\frr}{\Rightarrow}						%2-freccia

\deftwocell[black]{m : 2 -> 1}
\deftwocell[circle,black]{e : 0 -> 1}
\deftwocell[crossing]{s : 2 -> 2}
\deftwocell[white]{penta : 4 -> 1}
\deftwocell[white]{tria : 2 -> 1}
\deftwocell[white,rectangle]{exa : 3 -> 1}
\deftwocell[white,circle]{exa2 : 3 -> 1}
\deftwocell[white, rectangle]{inv : 2 -> 1}
\deftwocell[white]{g : 3 -> 1}
\deftwocell[gray]{alpha : 3 -> 1}
\deftwocell[gray, rectangle]{gamma : 3 -> 1}
\deftwocell[gray, rectangle]{tau : 2 -> 1}
\deftwocell[gray,lefthalfcircle]{l : 1 -> 1}
\deftwocell[gray,righthalfcircle]{r : 1 -> 1}

\deftwocell[dots]{d : 2 -> 2}
\deftwocell[text=\phi, white]{phi : 2 -> 2} 			%generic gate	phi
\deftwocell[text=\phi', white]{phi1 : 2 -> 2} 			%generic gate	phi'
\deftwocell[text=\psi, white]{psi : 2 -> 2} 			%generic gate	psi
\deftwocell[text=\chi_u, white]{up : 2 -> 6} 				%upcontex	
\deftwocell[text=\chi_d, white]{do : 6 -> 2} 				%downcontex	

\deftwocell[text=\phi, white]{phi3 : 3 -> 3} 			%generic gate	phi 3->3
\deftwocell[text=\phi, white]{phi23: 2 -> 3} 			%generic gate	phi 2->3
\deftwocell[text=\phi, white]{phi24: 2 -> 4} 			%generic gate	phi 2->4
\deftwocell[text=\phi', white]{phi32 : 3 -> 2} 			%generic gate	phi 3->2
\deftwocell[text=\phi', white]{phi42 : 4 -> 2} 			%generic gate	phi 4->2
\deftwocell[text=\psi, white]{psi4 : 4 -> 4} 			%generic gate	psi 4->4
\begin{document}

\maketitle

\begin{abstract}

A symmetric monoidal category is a category equipped with an associative and commutative (binary) product and an object which is the unit for the product. In fact, those properties only hold up to natural isomorphisms which satisfy some \emph{coherence conditions}. The \emph{coherence theorem} asserts the commutativity of all \emph{linear} diagrams involving the left and right unitors, the associator and the braiding.

We prove the coherence for symmetric monoidal categories using a constructive homotopical method based on higher dimensional rewriting. For that scope, we detail the convergence proof of  Lafont's string diagram rewriting system which presents the isomorphisms of these theories.
\end{abstract}

\section{Introduction}

As remarked by Mac Lane in \cite{MacNat}, the associative law $a(bc)=(ab)c$ implies the ``general associative law'' involving any iterate product of the same factors in the same order. When the associativity is given by an isomorphism, as in monoidal categories, it becomes necessary to impose some conditions to generalize this remark. This \emph{coherence condition} is given by the \emph{pentagonal identity}. Huet \cite{GH} noticed that this proof can be seen as an application of the Knuth-Bendix procedure \cite{KB}. This idea is explored further by Guiraud, Malbos and Mimram in \cite{GuirMalHom} where is shown how this procedure defines new coherence conditions in \emph{higher dimensional rewriting systems}.

As application of these methods, in this paper we propose an alternative proof of the coherence theorem for symmetric monoidal categories by means of \emph{string diagrams rewriting}. Moreover, the interested reader might find here a detailed version of the confluence proof proposed by Y. Lafont for the rewriting system given in \cite{LafBool}. Even if  this rewriting theory is characterized by a more complex theory of confluence (see \cite{GuirMalHD}, \cite{LafBool}, \cite{mim3D}) due to the presence of different compositions as in higher dimensional categories \cite{HO}, we are able to give a finite homotopy base from the confluence diagrams of \emph{critical peaks} as suggested by the method used by Squier for his theorem about string rewriting systems (\cite{Squier}, \cite{Squier2}, \cite{LafFin}, \cite{LafChurc}).

In a similar way, Guiraud and Malbos in \cite{GuirMalCohe} give a non-constructive proof of coherence (for both symmetric and braided monoidal categories) by means of polygraphs \cite{Bur} using some additional invariance results for \emph{Tietze equivalent} rewriting systems and under the assumption of the existence of a convergent \emph{diagram rewriting system} (the one given by Lafont).

In order to make constructive our proof, in this paper we detail the solutions of critical peaks and we exhibit the analogous of Kelly's result \cite{Kelly} in the symmetrical case. In fact, similarly as in the monoidal case the two coherence conditions guarantee the coherence of all five the critical peaks of the rewriting system, in our presentation of symmetric monoidal categories we are able to prove the coherence of all seventeen critical peaks of the system with five coherence conditions.

We conclude the proof by giving a procedure to use the confluence of Lafont's rewriting system in order to find a decomposition any confluence diagram into smaller diagrams which correspond to the initial coherence conditions of the theory.

\section{Background}

\subsection{Monoidal category}
\begin{defi}\label{MC}
A \emph{monoidal category} is a category $\C$ equipped with a (bi)functor $\pro : \C \times \C \fr \C$ and an object $e$ of $\C$ (that can be seen as a functor $\e : * \fr \C$ ) together with  three families of natural isomorphisms 
$$\alpha=\alpha_{xyz}: (x\pro y)\pro z \fr x\pro (y\pro z) ,$$ $$ \lambda=\lambda_x:  \e  \pro x \fr x ,\qquad  \rho=\rho_x : x \pro  \e \fr x ,$$
such that the following diagrams commute:
\begin{itemize}
\item \emph{pentagonal identity}: 

\begin{center}
\begin{tikzpicture}[node distance=1.6 cm, auto]
  \node (0)  {$((x\pro y ) \pro z)\pro t)$};
  \node (1) [right of=0, above of=0] {$(x \pro (y\pro z)) \pro t)$};
\node(1b)[right of =1]{};
  \node (2) [right of=1b] {$ x \pro ((y\pro z) \pro t$};
  \node (3) [below of=2, right of =2] {$x\pro (y\pro (z\pro t))$}; 
 \node(4)[below of =0, right of =0]{};
\node(5)[right of=4]{$(x \pro y ) \pro (z\pro t)$};

\draw [->] (0) to node{$ \alpha \pro \id$ } (1);
\draw [->] (0) to node[swap]{$\alpha $} (5);
\draw [->] (5) to node[swap]{$\alpha$} (3);
\draw [->] (1) to node{$ \alpha$} (2);
\draw [->] (2) to node{$\id \pro \alpha $} (3);
\end{tikzpicture}
\end{center}

\item \emph{triangular identity}:
\begin{center}
\begin{tikzpicture}[node distance=3 cm, auto]
  \node (0) {$(x \pro \e)\pro y$};
  \node (1) [right of=0] {$x\pro y$};
  \node (3) [below of=0, right of =0 , node distance =1.5cm ] {$x \pro (\e \pro y)$};

\draw [->] (0) to node{$\rho \pro \id $} (1);
\draw [->] (0) to node[swap]{$\alpha$} (3);
\draw [->] (3) to node[swap]{$\id \pro \lambda$} (1);
\end{tikzpicture}
\end{center}

\end{itemize}
\end{defi}

We note that the pentagonal and triangular identities correspond to the confluence of two critical peaks for the rewriting system given by $\alpha$, $\lambda$ and $\rho$. In fact, there are three other critical peaks which are handled by the following lemma, which is proved in section \ref{cohemonodia}.

\begin{teor}[Coherence of monoidal categories]\label{cohemc}\hfill \\
In an arbitrary monoidal category, every diagram made of $\lambda, \rho$ and $ \alpha$  is commutative.
\end{teor}

\begin{lemma}[Kelly \cite{Kelly}]\hfill \\
The commutativity of the following diagrams  is derivable from the pentagonal and the triangular identities:

\begin{center}
\resizebox{.8\width}{!}{
\begin{tikzpicture}[node distance=3 cm, auto]
  \node (0) {$(\e \pro x)\pro y$};
  \node (1) [right of=0] {$x\pro y$};
  \node (3) [below of=0, right of =0 , node distance =1.5cm ] {$\e \pro (x \pro y)$};

\draw [->] (0) to node{$\lambda \pro \id$} (1);
\draw [->] (0) to node[swap]{$\alpha$} (3);
\draw [->] (3) to node[swap]{$\lambda$} (1);
\end{tikzpicture}
\qquad
\begin{tikzpicture}[node distance=3 cm, auto]
  \node (0) {$(x \pro y)\pro \e$};
  \node (1) [right of=0] {$x\pro y$};
  \node (3) [below of=0, right of =0 , node distance =1.5cm ] {$x \pro (y \pro \e)$};

\draw [->] (0) to node{$\rho$} (1);
\draw [->] (0) to node[swap]{$\alpha$} (3);
\draw [->] (3) to node[swap]{$\id \pro \rho$} (1);
\end{tikzpicture}
\qquad
\begin{tikzpicture}[node distance=3 cm, auto]
  \node (0) {$\e \pro \e$};
  \node (1) [right of=0] {$ \e {\color{white} \pro \e}$};
\draw[->] (0.10) to node {$\lambda$} (1.170);
\draw[->] (0.350) to node[swap]{$\rho$} (1.190);
\end{tikzpicture}
}
\end{center}

\end{lemma}

\begin{defi}[Symmetric monoidal category]\label{SMC}\hfill \\
A \emph{symmetric monoidal category} is a monoidal category $\C$ equipped with an extra family of natural isomorphisms called braidings
$$ \sym= \sym_{x y} : x \pro  y \fr y\pro x$$
such that the following diagrams commute:
\begin{itemize}
\item the \emph{hexagonal identity}:
\begin{center}
\begin{tikzpicture}[node distance=1.3 cm, auto]
  \node (0) {$(x \pro y)\pro z$};
  \node (1) [right of=0, below of =0] {$x\pro (y \pro z)$};
\node(c)[right of=1, above of =1] {};
  \node (2) [below of=c, right of =c] {$(y \pro z) \pro x$};
  \node(3)[right of=2, above of =2]{$y\pro (z\pro x)$};
  \node (4) [right of=0, above of =0] {$(y \pro x) \pro z$};
  \node (5) [above of=c, right of =c] {$y \pro (x \pro z)$};

\draw [->] (0) to node[swap]{$\alpha$} (1);
\draw [->] (1) to node[swap]{$\sym$} (2);
\draw [->] (2) to node[swap] {$\alpha$}(3);

\draw [->] (0) to node{$\sym \pro \id$} (4);
\draw[->](4) to node {$\alpha$}(5);
\draw[->](5) to node {$\id \pro \sym$}(3);

\end{tikzpicture}
\end{center}
\item \emph{ involutivity} of $\sym$:
\begin{center}
\begin{tikzpicture}[node distance=2 cm, auto]
  \node (0) {$x \pro y$};
  \node (1) [right of=0] {$x\pro y$};
  \node (3) [below of=0, right of =0 , node distance =1cm ] {$y \pro x$};

\draw[-] (0.5) to node {} (1.175);
\draw[-] (0.355) to node{} (1.185);
\draw [->] (0) to node[swap]{$\sym$} (3);
\draw [->] (3) to node[swap]{$\sym$} (1);
\end{tikzpicture}
\end{center}
\end{itemize}
\end{defi}

%\begin{example}
%Any cartesian category, has a structure of symmetric monoidal category, where $\pro$ is the cartesian product $\times$ and $\sym: x \times y \fr y \times x$ is the exchange of components.
%\end{example}

\begin{defi} A diagram is \emph{linear}\footnote{In \cite{Mac}, Mac Lane calls this a \emph{formal} diagram} if  every variable appears at most once on each vertex. \end{defi}

\begin{teor}[Coherence of symmetric monoidal categories, \cite{MacNat}]\label{cohesmc}\hfill \\
In an arbitrary symmetric monoidal category, every linear diagram made of $\lambda, \rho, \alpha$ and $ \sym$ is commutative.
\end{teor}

In this paper, in order to have a parallelism between natural isomorphisms and the set of the rewriting system we study, we use an alternative equivalent definition of the symmetric monoidal categories (see \cite{alt}) given given by adding extra (superfluous) family of natural isomorphisms $\gamma$ (which represent the so called \emph{parallel associativity}\cite{jova}) and imposing the commutativity of the two following confluence diagrams which decompose the confluence diagram of the hexagonal identity by means of $\gamma$:

\begin{defi}[Symmetric monodial category]\label{SMC2}
A \emph{symmetric monoidal category} is a monoidal category $\C$ equipped with two extra families of natural isomorphisms respectively called braiding and parallel associativity
$$ \sym= \sym_{x y} : x \pro  y \fr y\pro x \quad \mbox{ and }\quad \gamma= \gamma_{x y z} : x \pro ( y \pro z)\fr y\pro (x\pro z)$$
such that the following diagrams commute:
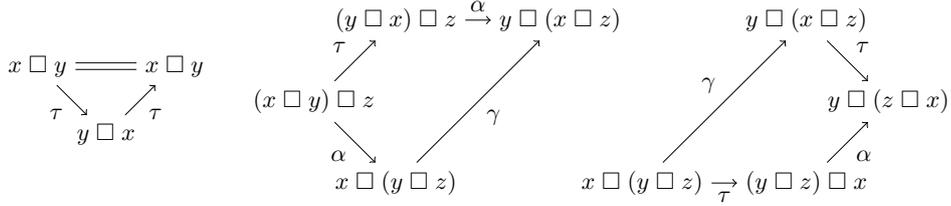
\begin{center}
\resizebox{.9\width}{!}{
\begin{tikzpicture}[node distance=2 cm, auto]
  \node (0) {$x \pro y$};
  \node (1) [right of=0] {$x\pro y$};
  \node (3) [below of=0, right of =0 , node distance =1cm ] {$y \pro x$};
\node(4)[below of=3, node distance =1cm]{~};
\draw[-] (0.5) to node {} (1.175);
\draw[-] (0.355) to node{} (1.185);
\draw [->] (0) to node[swap]{$\sym$} (3);
\draw [->] (3) to node[swap]{$\sym$} (1);
\end{tikzpicture}
\quad
\begin{tikzpicture}[node distance=1.2 cm, auto]
  \node (0) {$(x \pro y)\pro z$};
  \node (1) [right of=0, above of =0] {$(y\pro x) \pro z$};
  \node (1b)[right of =1]{};
\node(2)[right of=1b] {  $y\pro (x\pro z)$ };
  \node (4) [right of=0, below of =0] {$x \pro (y \pro z)$};
\draw [->] (0) to node{$\sym$} (1);
\draw [->] (1) to node{$\alpha$} (2);
\draw [->] (0) to node[swap]{$\alpha$} (4);
\draw[->](4) to node [swap] {$\gamma$}(2);
  \node(c1)[right of=2, below of =2] {};
  \node(c2)[right of=c1] {};
  \node (22) [above of=c2, right of =c2] {$y\pro (x\pro z)$};
  \node(32)[right of=22, below of =22]{$y\pro (z\pro x)$};
  \node (42) [left of=c2, below of =c2] {$x \pro (y \pro z)$};
  \node (52) [below of=c2, right of =c2] {$(y \pro z) \pro x$};
\draw [->] (22) to node {$\sym$}(32);
\draw[->](42) to node [swap] {$\sym$}(52);
\draw[->](52) to node [swap] {$\alpha$}(32);
\draw [->] (42) to node {$\gamma$}(22);
%\draw [gray, dashed](2) -- (22);
%\draw [gray, dashed](4) -- (42);
\end{tikzpicture}
}
\end{center}
\end{defi}

\subsection{String diagrams}

We now recall some basic notions in string diagram rewriting. In this paper we use a particular setting, which can be called \emph{monochrome string diagrams}, since we use no labels on backgrounds neither on strings. For an more general introduction to string diagrams, see \cite{string2} or P. Selinger paper \cite{selinger2010survey}. % while some historical r J. Baez's notes \cite{baez}.

Given $p$ and $q$ natural numbers, a diagram $\phi : p \frr q$ with $p$ \emph{inputs} and $q$ \emph{outputs} is pictured as follows:
$$
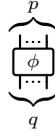
\overbrace{\underbrace{
\twocell{d *1 phi *1 d}}_{q}}^{p}$$
Diagrams may be composed in two different ways. If $\phi : p \frr q$ and $\phi': p' \frr q'$ are diagrams, we define:
\begin{itemize}

\item \emph{sequential} composition: if $q=p'$, the diagram $\phi' \circ \phi : p\frr q'$ corresponds to usual composition of maps.

This composition is associative with unit $\id_p: p\frr p$ for each $p\in \N$. In other words, we have $\phi \circ \id_p= \phi= \id_q \circ \phi$. The\emph{ identity diagram} $id_p$ is pictured as follows: $\underbrace{
\twocell{d}}_{p}$

\item \emph{parallel}  composition: the diagram $\phi * \phi': p+p' \frr q+q'$ is always defined.
This composition is associative with unit $\id_0: 0 \frr 0$. In other words, we have $\id_0 * \phi=\phi=\phi * \id_0$. This $\id_0$ is called the \emph{empty diagram}.
\end{itemize}
These two compositions are respectively represented as follows:
$$
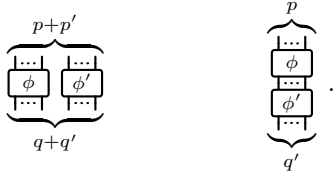
\overbrace{\underbrace{
\twocell{(d *0 d)*1( phi *0 phi1)*1(d *0 d)}}_{q+q'}}^{p+p'}
\qquad \qquad \qquad
\overbrace{\underbrace{
\twocell{d *1 phi *1 d *1 phi1 *1 d}}_{q'}}^{p}\; .
$$

Our two compositions satisfy the \emph{interchange rule}: if $\D:p\frr q$ and $\D' : p'\frr q'$, so $ (\id_q * \phi' ) \circ (\phi * \id_{p'})= \phi * \phi' =(\phi * \id_{q'})\circ (\id_p * \phi')$ that corresponds to the following picture:

$$\twocell{(d *0 d) *1 (phi *0 2) *1 (2 *0 phi1) *1 (d *0 d)}=\twocell{(d *0 d) *1 (phi *0 phi1) *1 (d *0 d)}=\twocell{(d *0 d) *1  (2 *0 phi1) *1 (phi *0 2) *1 (d *0 d)}$$

 Monochrome string diagrams can be interpreted as morphisms in a $\mathbf{PRO}$, that is a strict monoidal category whose objects are natural numbers and whose product  on objects is addition. %To be coherent with the cellular notation we use in next sections, diagrams represent 2-arrows in the $\mathbf{2-PRO}$ obtained by suspension of a regular  $\mathbf{PRO}$ (see \cite{suspension}). 

\begin{defi}[Signature]
A \emph{signature} $\sig$ is a finite set of \emph{atomic diagrams} (or \emph{gates type}). Given a signature, a diagram $\phi : p \frr q\in \sig$ is a morphism in the  $\mathbf{PRO}$  freely generated by $\sig$, i.e. by the two compositions and identities. A \emph{gate} is an occurrence of an atomic diagram. %, we note $g: \alpha$ if $g$ is an occurrence of $\alpha \in \sig$.

\end{defi}

%\begin{defi}
%We say that $\phi$  is a \emph{subdiagram} of $\phi'$ whenever there exist $ \psi_u,\psi_d\in \sig^*$ and $k,k' \in \N$ such that $\phi'=\psi_d \circ (\id_\Gamma *\phi*\id_{\Delta})\circ \psi_u $.
%\end{defi}

%\nota Given $\phi\in \sig^*$ and $\sig' \subseteq \sig$, we write $|\phi|_{\sig'}$ the number of gates in $\phi$ with gate type $\alpha \in \sig'$.

%\begin{defi}
%We call \emph{horizontal} a diagram $\phi$ generated by parallel composition (and identities) only in $\sig^*$. It is \emph{elementary} if $|\phi|_\sig=1$.
%\end{defi}
%

%\pastesection{diag}

\subsection{Diagram rewriting}
\begin{defi}[Diagram Rewriting System]
A \emph{diagram rewriting system} is  a couple $(\sig, \R)$ given by a signature $\sig$ and a set $\R$ of rewriting rules of the form $\rew \phi {\phi'}$ where $\phi, \phi': p \frr q $ are diagrams in $\sig^*$. 

Moreover, we allow each rewriting  rules under any context, that is, if  $\xymatrix@C=1em{ \phi \ar@3[r] & \phi'}$ in $\R$ then, for every $\chi_u , \chi_d\in \sig^*$. We respectively represent a rewriting rule and its application under a context as follows:
$$
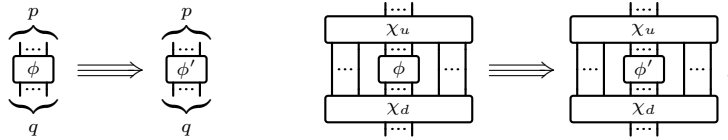
\overbrace{\underbrace{
\twocell{d *1 phi *1 d}}_q}^p
\xymatrix{\ar@3[r]&}
\overbrace{\underbrace{
\twocell{d *1 phi1 *1 d}}_q}^p
\qquad \qquad
\twocell{d *1 up *1 (2 *0 d *0 2) *1 (d *0 phi *0 d ) *1 ( 2 *0 d *0 2) *1 do *1 d}
\xymatrix{\ar@3[r]&}
\twocell{d *1 up *1 (2 *0 d *0 2) *1 (d *0 phi1 *0 d ) *1 ( 2 *0 d *0 2) *1 do *1 d} \; .$$

 We say that $\psi$ \emph{reduces}, or \emph{rewrites}, to $\psi'$  (denoted $\xymatrix@C=1em{\psi \ar@3[r]^{*} & \psi'}$) if there is a \emph{rewriting path} $P: \xymatrix@C=1em{\psi=\psi_0 \ar@3[r] & \psi_1 \ar@3[r] & \dots  \ar@3[r] & \psi_n=\psi'}$.
\end{defi}

We here recall some classical notions in rewriting:

\begin{itemize}
\item A diagram $\phi$ is \emph{irreducible} if there is no $\phi'$ such that $\xymatrix@C=1em{\phi \ar@3[r] & \phi'}$;

\item A rewriting system \emph{terminates} if there is  no infinite rewriting sequence;

\item A rewriting system is \emph{confluent} if for all $\phi_1,\phi_2$ and $\D$ such that $\xymatrix@C=1em{\phi \ar@3[r] & \phi_1}$ and $\xymatrix@C=1em{\phi \ar@3[r] & \phi_2 }$, there exists $\phi' $ such that $\xymatrix@C=1em{\phi_1 \ar@3[r]^{*} & \phi'} $  and $\xymatrix@C=1em{\phi_2 \ar@3[r]^{*} & \phi'}$;

\item A rewriting system is \emph{convergent} if it terminates and is confluent.

\end{itemize}

\begin{prop}[Newman lemma \cite{new}]\label{termconf}
A terminating rewriting system is confluent if and only if every critical peak is solvable.
\begin{proof} 
The left-to-right implication is trivial by definition of confluent rewriting system and critical peak.

In order to prove the converse we have to prove that any conflict is solvable. 
%\corr{
%If a conflict has no overlapping, it is locally confluent. Otherwise it is a critical peak with source $w$ or a critical pair with source $uwv$ where $w$ is a critical peak and $u,v,w\in\s^*$ and so it is solvable. }
{Lets consider a conflict $(P_1,P_2)$ and two maximal rewriting paths $P_1'P_1: \rew{\phi}{\hat \psi_1}$ and $P_2'P_2: \rew{\phi}{\hat \psi_2}$.
If $(Q_1P_1,Q_2P_2)$ is a solution of the conflict $(P_1,P_2)$ with $Q_1P_1,Q_2P_2: \rew \phi {\phi'}$  and we have $\hat \psi_1=\phi'$ or $\hat \psi_2=\phi'$ the proposition is proved. 
Otherwise, we consider the maximal rewriting paths $\hat QQ_1:\rew {\psi_1}{\hat \psi'}$ and $\hat QQ_2: \rew{\psi_2}{\hat \phi'}$. 

\centerline{
$$\xymatrix{
& &\psi_1 \ar@3[dr]_{Q_{1}}^{*} \ar@3@/^/[rrr]^{~~~~~~~~~~~~P_1'~~~~~~~~*} \ar@3@/^/[drrr]^{*}_{\hat Q Q_1}&&&\hat \psi_1  \\
\phi \ar@3[drr]^{P_2}\ar@3[urr]_{P_1}& && \phi' \ar@3[rr]^{~~\hat Q~~~*} && {\hat \phi'} \ar@3@{=}[u] \ar@3@{=}[d]\\
& &\psi_2 \ar@3[ur]^{Q_2}_{*} \ar@3@/_/[rrr]_{~~~~~~~~~~~~P_2'~~~~~~~~*} \ar@3@/_/[urrr]^{*}_{\hat Q Q_2}& && \hat \psi_2 
}$$}
\noindent By induction over the length of $P_1'P_1$ and $P_2'P_2$, we are able to prove that $\hat \psi_1=\hat \phi'=\hat \psi_2$.
}
\end{proof}
\end{prop}

Without providing full details about \emph{polygraphical} construction of higher dimensional categories by means of generators and relations \cite{Bur}, we refer to  string diagrams as $2$-cells and to rewriting paths as $3$-cells. Moreover, whenever we have two \emph{parallel} $3$-cells (i.e. two rewriting paths with same source and same target), we are able to define a $4$-cell of which this pair is border. 

Even if the other cells are implicitly ($2$-cells up-down) or explicitly ($3$-cells) oriented, we do not give any orientation for $4$-cells as well as, in equalities, is not usual to give one. For this reason, this initial setting we consider is called $(4,3)$-polygraph \cite{maxime}. In general, we refer to a \emph{$(n,m)$-polygraph} as a collection of sets of $k$-cells for any $k\leq m$ such that there is no orientation for the $k$-cells with $k>m$.

\begin{lemma}[Refinement of Newman lemma]\label{pav}\hfill \\
Given a convergent diagram rewriting system $(\sig, \R)$ and a $4$-cell for each critical peak solution, then there is a composed $4$-cell for any pair $(P,Q)$ of non-equal parallel $3$-cells.
\end{lemma}
\begin{proof} 
It suffice to remark that in the construction of the proof of Proposition \ref{termconf}, whenever we have the the solution a critical perk appearing in the decomposition of two parallel $3$-cells, we can fill its relative confluence diagram by the corresponding $4$-cells. 
Then, the proof is given by induction over the number of critical peak solutions occurring in a confluence diagram decomposition.
\end{proof}

\section{Coherence for monoidal categories via rewriting}\label{cohemonodia}

In this section we focus on the convergence proof of two diagram rewriting systems $\M$ and $\Sim$ that presents functors and natural isomorphisms of respectively monoidal category and symmetric monoidal categories and we prove their confluence.
We give the full proof  the proof of the confluence of $\Sim$, whose main passages were already proved in \cite{LafBool}. In particular we analyze the process to determine the 68 critical peaks of this rewriting system and we exhibit their confluence diagrams.

\begin{defi}\label{SimDef}
Let $\Sim$ be the rewriting system defined on the signature given by the following gates
$$
\twocell{m} \qquad \qquad
\twocell{e} \qquad \qquad
\twocell{s}
$$
 and with the following rules:
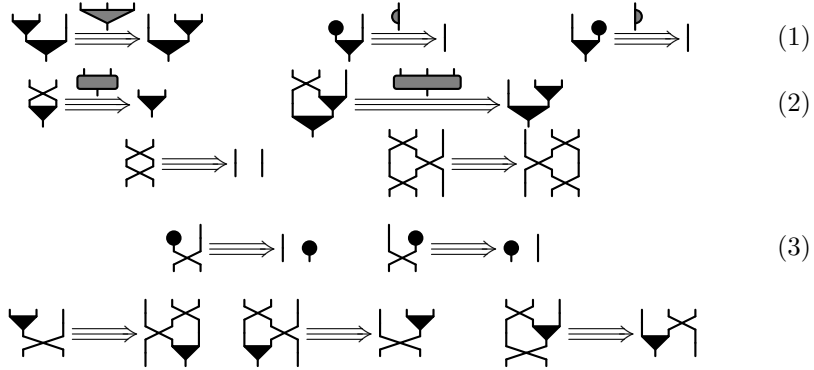
\vspace{-.3cm}
\begin{equation}\label{monr}{
\xymatrix{ \twocell{( m *0 1) *1 m} \ar@3@3[r]^{\twocell{alpha}} & \twocell{(1 *0 m)*1 m}} \qquad \qquad
\xymatrix{ \twocell{(e *0 1)*1 m} \ar@3[r]^{\twocell{ l }} & \twocell{1}} \qquad \qquad
\xymatrix{ \twocell{(1 *0 e )*1 m} \ar@3[r]^{\twocell{ r }} & \twocell{1}}
}\end{equation}
\vspace{-.6cm}
\begin{equation}\label{newr}{
\xymatrix{ \twocell{ s *1 m} \ar@3[r]^{\twocell{ tau }} & \twocell{m}} \qquad \qquad
\xymatrix{ \twocell{(s *0 1) *1 ( 1 *0 m) *1 m} \ar@3@3[rr]^{\twocell{gamma}} && \twocell{(1 *0 m)*1 m}} \qquad \qquad
}\end{equation}
\vspace{-.6cm}
\begin{equation}\label{manr}
\begin{gathered}
\xymatrix{ \twocell{s *1 s } \ar@3[r] & \twocell{2}} \qquad \qquad
\xymatrix{ {\twocell{(s *0 1)*1 (1 *0 s)*1 (s *0 1)}} \ar@3[r] & {\twocell{ (1 *0 s)*1 (s *0 1)*1 (1 *0 s) }}}\\
\xymatrix{ \twocell{(e *0 1) *1 s } \ar@3[r] & \twocell{(1 *0 e) *1 2}} \qquad
\xymatrix{ \twocell{(1 *0 e) *1 s} \ar@3[r] & \twocell{(e *0 1) *1 2}} \\
\xymatrix{ {\twocell{(m *0 1)*1 s}} \ar@3[r] & {\twocell{ (1 *0 s)*1 (s *0 1)*1 (1 *0 m) }}} \quad
\xymatrix{ {\twocell{(s *0 1)*1 (1*0 s) *1 (m*0 1)}} \ar@3[r] & {\twocell{ (1 *0 m)*1 s }}} \qquad 
\xymatrix{ \twocell{(s *0 1) *1 (1 *0 m) *1 s } \ar@3[r] & \twocell{(1 *0 s ) *1 (m *0 1)} }
\end{gathered}
\end{equation}
By restriction, we define the rewriting system $\M$ as the one given by the signature $\{ \twocell{m}, \twocell{e} \}$ and the set of rules (\ref{monr}).
\end{defi}

A string diagram in $\Sim$ should be read labeling with free variables the input strings (keeping the intuition that $\twocell s$ is a string crossing), by $\e$ the output string of a $\twocell e$-gate and by $x\pro y$ the output string of a $\twocell m$-gate with input string $x,y$. By means of example, the interpretation of the diagrams $ \twocell{(s *0 e)*1(1 *0 m)*1 m}$ and $\twocell{(1*0 m) *1 s}$ respectively are $y \pro(x \pro e)$, and the pair $(y\pro z, x)$.

\begin{oss}
Structural rules do not change digram interpretation: if $\xymatrix@C=.8em{ \phi \ar@3^{*}[r] & \phi'}$ is a rewriting path made of structural rules, then the interpretation of $\phi$ and $\phi'$ is the same.
\end{oss}

\begin{oss}The set of rule can be derived by the following set of equivalencies:
$$\begin{gathered}
\twocell{s *1 s } = \twocell{2} \qquad \qquad
\twocell{(s *0 1)*1 (1 *0 s)*1 (s *0 1)} =\twocell{ (1 *0 s)*1 (s *0 1)*1 (1 *0 s) }\qquad \qquad
\twocell{(e *0 1) *1 s } = \twocell{(1 *0 e) *1 2} \qquad \qquad
\twocell{(m *0 1)*1 s} = \twocell{ (1 *0 s)*1 (s *0 1)*1 (1 *0 m) }\\
\twocell{( m *0 1) *1 m} = \twocell{(1 *0 m)*1 m} \qquad \qquad
\twocell{(e *0 1)*1 m} = \twocell{1} \qquad \qquad
\twocell{ s *1 m} = \twocell{m}.
\end{gathered}$$
We have divided this set of rules into three subsets: the subset (\ref{monr}) consisting of \emph{monoidal rules} that are the same rules as $\M$ which correspond to  associativity and left and right unitor,  the subset (\ref{newr}) consisting of rules \twocell{tau}, corresponding to the natural transformation $\tau$, and \twocell{gamma} which is the natural transformation of parallel associativity $\gamma$ and the subset of \emph{structural rules} (\ref{manr}) required to to manage resources and their interactions.
\end{oss}

\begin{oss}
Lafont's \cite{LafBool} and Guiraud-Malbos' \cite{GuirMalCohe} presentations of  polygraphs representing symmetric monoidal categories differ in the orientations of the rules \ref{manr}. In particular, the two systems differ for the orientation of 
\vspace{-.3cm}
$$\mbox{Lafont: }\xymatrix{ {\twocell{(s *0 1)*1 (1*0 s) *1 (m*0 1)}} \ar@3[r] & {\twocell{ (1 *0 m)*1 s }}}\qquad \mbox{Guiraud-Malbos: } \xymatrix{{\twocell{ (1 *0 m)*1 s }} \ar@3[r] & {\twocell{(s *0 1)*1 (1*0 s) *1 (m*0 1)}}}$$
\vspace{-.3cm}

and the presence in Lafont's system of the extra structural rule \resizebox{.8\width}{!}{$\xymatrix{ \twocell{(s *0 1) *1 (1 *0 m) *1 s } \ar@3[r] & \twocell{(1 *0 s ) *1 (m *0 1)} }$} required to recover confluence.

 Guiraud-Malbos rewriting system is not confluent; more precisely, the critical peak of $\twocell{(1*0 s)*1 (1 *0 m )*1 s}$ is not solvable: 
\resizebox{.8\width}{!}{$$\xymatrix%@C=.5em@R=.8em
{ \twocell{(1*0 s)*1 (s *0 1 )*1 (1*0 s)*1 (m*0 1)} & \twocell{(1*0 s)*1 (1 *0 m )*1 s} \ar@3[r] \ar@3[l] & \twocell{(1 *0 m )*1 s}\ar@3[r] &\twocell{(s*0 1)*1 (1 *0 s)*1 (m*0 1)}} 
$$
}

On the other hand, confluence is guaranteed for rewriting paths made of  string diagrams with only one output, which are the string diagrams representing functors in symmetric monoidal categories.
 Moreover their method is not constructive and they are able to prove the coherence theorem using the existence of a Tietze-equivalent \cite{tietz} confluent presentation, i.e. the one given by Lafont.

\end{oss}

Here, for completeness, we recall the proof of the confluence of  $\Sim$.

\begin{prop}[\cite{LafBool}]\label{termsim}
The rewriting system $\Sim$ is terminating.
\end{prop}

\begin{proof} 
In order to prove termination we interpret every diagram $\phi: n\frr m \in  \Sim^*$ with a monotone function $[\phi]:\N^n\fr \N^m$. These have well-founded partial order induced by product order on $\N^p$, i.e. $\bar x =(x_1, \dots,x_p)\leq(y_1, \dots y_p)=\bar y$ whenever $x_1\leq x_1 \land \dots \land x_p \leq y_p$:
$$f,g: \N^{*p}\fr \N^{*q} \mbox{ then } f< g \mbox { iff } f(\bar x)< g(\bar x) \mbox{ for all } \bar x \in \N^{*p}.$$

Each gate of the signature can be interpreted as follows:
$$[\;\twocell{s}\;] (x,y)\fr (y, x+y) \qquad [\; \twocell{m}\; ] (x,y)\fr 2x+y \qquad [\twocell{e}](\emptyset )\fr 1. $$

This allow us to associate to any $3$-cell $\xymatrix@C-1pc{ \phi \ar@3[r] & \psi}$ two monotone maps  $[\phi]$ and $[\psi]$ such that $[\phi]> [\psi]$:
\vspace{-.1cm}
$$
\Big[\; \twocell{( m *0 1) *1 m}\; \Big](x,y,z)= 4x+2y+z  > 2x+2y+z =\Big[\;\twocell{(1 *0 m)*1 m}\; \Big](x,y,z) ,$$
$$
\Big[\;\twocell{(e *0 1)*1 m}\;\Big](x)=x+1 > x= \Big[\;\twocell{1}\;\Big](x), \qquad \qquad
\Big[\;\twocell{(1 *0 e )*1 m}\;\Big] (x)= x+1 > x = \Big[\;\twocell{1}\;\Big](x),$$
$$
\Big[\; \twocell{s *1 s }\;\Big] (x,y)= (2x+y,x+y)> (x,y)=\Big[\;\twocell{2}\;\Big](x,y), $$
$$
\Bigg[\;{\twocell{(s *0 1)*1 (1 *0 s)*1 (s *0 1)}}\;\Bigg](x,y,z)=(2x+y+z, x+y,x)> (x+y+z,x+y,x)=\Bigg[\;{\twocell{ (1 *0 s)*1 (s *0 1)*1 (1 *0 s) }}\;\Bigg],$$
$$
\Big[\; \twocell{(e *0 1) *1 s }\;\Big](x)= (x+1,1) > (x,1)=\Big[\;\twocell{(1 *0 e) *1 2}\;\Big] (x),$$
$$
\Big[\;\twocell{(1 *0 e) *1 s}\;\Big](x)= (x+1,x)> (1,x)=\Big[\; \twocell{(e *0 1) *1 2}\;\Big](x),$$
$$
\Bigg[\; \twocell{(m *0 1)*1 s}\;\Bigg](x,y,z)= (2x+y+x, 2x+y) > (x+y+z,2x+y)=\Bigg[\; \twocell{ (1 *0 s)*1 (s *0 1)*1 (1 *0 m) }\;\Bigg](x,y,z),$$
$$
\Bigg[\; \twocell{(s *0 1)*1 (1*0 s) *1 (m*0 1)} \;\Bigg](x,y,z)= (3x+2y+z,x) > (x+2y+z,x) = \Bigg[\;\twocell{ (1 *0 m)*1 s }\;\Bigg](x,y,z),$$
$$
\Bigg[\; \twocell{(s *0 1) *1 (1 *0 m) *1 s }\;\Bigg](x,y,z)= (3x+y+z,x+y) > (2x+y+z,y)=\Bigg[\; \twocell{(1 *0 s ) *1 (m *0 1)}\;\Bigg](x,y,z),$$
$$
\Bigg[\;\twocell{ s *1 m}\;\Bigg](x,y)= 3x+2y > 2x+y=\Bigg[\;\twocell{m}\;\Bigg](x,y),$$
$$
\Bigg[\; \twocell{(s *0 1) *1 ( 1 *0 m) *1 m}\;\Bigg](x,y,z)= 4x+2y+z > 2x+2y+z=\Bigg[\; \twocell{(1 *0 m)*1 m}\;\Bigg](x,y,z).
$$
By the compatibility of the order with sequential and parallel composition, this suffice to prove that, for any couple of  diagrams, $[\phi] > [\psi]$ holds if $\rew \phi \psi$. Since there exists no infinite decreasing suite of monotone maps on positive integers, infinite reduction paths can not exist.
\end{proof}

\begin{cor}\label{monter}
The rewriting system $\M$ is terminating.
\end{cor}

In order to prove $\Sim$ (and $\M$) confluence we need to prove some additional properties:

\begin{defi}[Global conflict]
In a rewriting system $(\sig, \R)$, a \emph{global conflict} (or \emph{indexed critical pair} \cite{GuirMalHD}) is a family of $2$-cells source of the same critical peak $(P_1,P_2)$ containing a proper subdiagram $\phi$ in which no gate is rewritten by  $P_1$ nor $P_2$. A $2$-cell in a global conflict is a \emph{reduced global conflict} if this subdiagram $\phi$ is irreducible.
\end{defi}

By this definition, considering the rewriting rules in $\Sim$ we have the following
\begin{prop}[$\Sim$'s global conflicts]\label{simglob}
In $\Sim$, there are six global conflicts 
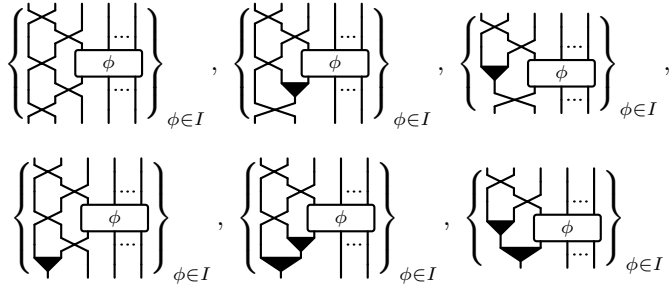
\vspace{-.2cm}
$$\begin{Bmatrix} \twocell{(s *0 3) *1(1 *0 s *0 d) *1 (s *0 phi3) *1 (1 *0 s *0 d) *1 (s *0 3)} \end{Bmatrix}_{\phi \in I}, \; 
\begin{Bmatrix} \twocell{(s *0 3) *1(1 *0 s *0 d) *1 (s *0 phi3) *1 (1 *0 m *0 d) *1 (s *0 2)} \end{Bmatrix}_{\phi \in I},  \; 
\begin{Bmatrix} \twocell{(s *0 3) *1(1 *0 s *0 d) *1 (m *0 phi3) *1 (s *0 d)} \end{Bmatrix}_{\phi \in I},$$
$$\begin{Bmatrix} \twocell{(s *0 3) *1(1 *0 s *0 d) *1 (s *0 phi3) *1 (1 *0 s *0 d) *1 (m *0 3)} \end{Bmatrix}_{\phi \in I}, \; 
\begin{Bmatrix} \twocell{(s *0 3) *1(1 *0 s *0 d) *1 (s *0 phi3) *1 (1 *0 m *0 d) *1 (m *0 2)} \end{Bmatrix}_{\phi \in I}, \; 
\begin{Bmatrix} \twocell{(s *0 3) *1(1 *0 s *0 d) *1 (m *0 phi3) *1 (m *0 d) } \end{Bmatrix}_{\phi \in I}$$\vspace{-.2cm}
where $I=\begin{Bmatrix}\phi \in \Sim^* | \: \phi= \twocell{ (1 *0 d) *1 phi3 *1 (1 *0 d)}: 1+n \frr 1+m \end{Bmatrix}$.
\end{prop}

\begin{lemma}\label{reducedglob}
An irreducible $2$-cells of $\Sim$ have one of the following four class of shapes:
\vspace{-.1cm}$$
\twocell{(d) *1 (e *0 phi)*1 (1 *0 d)} \;, \quad
\twocell{(1 *0 d) *1 (1 *0 phi)*1 (1 *0 d)}\;, \quad  
\twocell{(1 *0 d) *1 (1 *0 phi23 )*1 (m *0 d)}\; , \quad
\twocell{(1 *0 d) *1 (1 *0 phi23)*1 (s *0 d) *1 (1*0 phi32) *1 (1*0 d)}\;.  $$
\begin{proof} 
\vspace{-.2cm}
In order to prove the lemma, we observe how the rewriting rules act on a $2$-cell belonging in one of the above class: a $2$-cell of the first two classes can be rewritten only in a $2$-cell of the same class, a $2$-cell of the third  one can be rewritten in one of the first three classes while a $2$-cell of the fourth in any of these classes.

Moreover, these are the only possible configurations where the number of gates we pass through following the left-most path going from the left-most input to the left-most output of a $2$-cell is less than two\footnote{we remind the reader that \twocell s is not a wire crossing but a gate with two inputs and two outputs, this is $\twocell s = {\raisebox{-2.50pt}{\begin{tikzpicture} \begin{scope} [ x = 10pt, y = 10pt, join = round, cap = round, thick, black, solid, -]  \draw (0.00,0.00)--(0.00,-0.25) (1.00,0.00)--(1.00,-0.25) ; \draw[lightgray] (0.00,0.50)--(1.00,0.00) (1.00,0.50)--(0.00,0.00) ;\draw [rounded corners = 1pt] (-0.25,0) rectangle (1.25,0.50) ; \draw (0.00,0.75)--(0.00,0.50) (1.00,0.75)--(1.00,0.50) ; \end{scope} \end{tikzpicture}}} $}. In fact, if this condition is not satisfy, we note that a $2$-cell can not be irreducible. In fact, by induction over the number of gates in a $2$-cell, the possible shapes with irreducible context are the following:
\vspace{-.1cm}$$
\twocell{(1 *0 d) *1 (1 *0 phi24 )*1 (m *0 1 *0 d) *1 (m*0 2)}\; , \;
\twocell{(1 *0 d) *1 (1 *0 phi23 )*1 (s *0 d) *1 (m *0 2)}\; , 
\twocell{(2 *0 d) *1 (2 *0 phi23 )*1 (m  *0 1 *0 d) *1 (s *0 2)*1 (1 *0 phi32) *1 (1 *0 d)} , \;
\twocell{(1 *0 d) *1 (1 *0 phi23)*1 (s *0 d) *1 (s *0 2) *1 (1 *0 phi32) *1 (1 *0 d)}\; , \;
\twocell{(1 *0 d) *1 (1 *0 phi24)*1 (s *0 1 *0 2) *1 (1 *0 s *0 d) *1 (m *0  phi32) *1 (1 *0 d)}\; , \;
\twocell{(1 *0 d) *1 (1 *0 phi24)*1 (s *0 1 *0 2) *1 (1 *0 s *0 d) *1 (s *0 1 *0 2) *1 (1 *0 phi42) *1 (1 *0 d)}\; , \;
{\raisebox{-25.00pt}
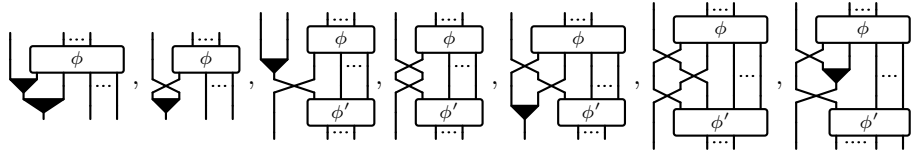
{\begin{tikzpicture} \begin{scope} [ x = 10pt, y = 10pt, join = round, cap = round, thick, black, solid, -] \draw (0.00,5.25)--(0.00,5.00) (2.00,5.25)--(2.00,5.00) (3.00,5.25)--(3.00,5.00) ; \draw (0.00,0.00)--(0.00,-0.25) (1.50,0.00)--(1.50,-0.25) (3.00,0.00)--(3.00,-0.25) (4.00,0.00)--(4.00,-0.25) ;  \draw [dash pattern = on 0.25pt off 2pt] (2.25,5.00)--(2.75,5.00) ; \draw (0.00,5.00)--(0.00,4.25) (2.00,5.00)--(2.00,4.75) (3.00,5.00)--(3.00,4.75) ;  \draw [rounded corners = 1pt, fill = white] (0.75,3.75) rectangle (4.25,4.75) ; \node at (2.50,4.25) {$\scriptstyle \phi$} ; \draw (0.00,4.25)--(0.00,3.50) (1.00,3.75)--(1.00,3.50) (2.00,3.75)--(2.00,3.25) (3.00,3.75)--(3.00,3.25) (4.00,3.75)--(4.00,3.25) ; \draw (0.00,3.50)--(1.00,3.00) (1.00,3.50)--(0.00,3.00) ;  \draw (0.00,3.00)--(0.00,2.50) (1.00,3.00)--(1.00,2.75) (2.00,3.25)--(2.00,2.75) (3.00,3.25)--(3.00,2.50) (4.00,3.25)--(4.00,2.50) ;  \draw [fill = black] (1.00,2.75)--(2.00,2.75)--(1.50,2.25)--cycle ; \draw [dash pattern = on 0.25pt off 2pt] (3.25,2.50)--(3.75,2.50) ; \draw (0.00,2.50)--(0.00,2.00) (1.50,2.25)--(1.50,2.00) (3.00,2.50)--(3.00,1.75) (4.00,2.50)--(4.00,1.75) ; \draw (0.00,2.00)--(1.50,1.50) (1.50,2.00)--(0.00,1.50) ;  \draw (0.00,1.50)--(0.00,0.75) (1.50,1.50)--(1.50,1.25) (3.00,1.75)--(3.00,1.25) (4.00,1.75)--(4.00,0.75) ;  \draw [rounded corners = 1pt, fill = white] (1.25,0.25) rectangle (4.25,1.25) ; \node at (2.75,0.75) {$\scriptstyle \phi'$} ;  \draw (0.00,0.75)--(0.00,0.00) (1.50,0.25)--(1.50,0.00) (3.00,0.25)--(3.00,0.00) (4.00,0.25)--(4.00,0.00) ;  \draw [dash pattern = on 0.25pt off 2pt] (1.88,0.00)--(2.63,0.00) ;  \end{scope} \end{tikzpicture}}} \; ,  $$\vspace{-.1cm}
 which are all reducible.
\end{proof}  
\end{lemma}

\begin{lemma}[Global conflict solution]\label{glob}
In a confluent rewriting system $(\sig, \R)$ all critical peaks in a global conflict are confluent if and only if all reduced global conflicts are.
\begin{proof} 
The left-to-right implication is trivial. In order to prove the right-to-left implication it suffice to remark that every $2$-cell $ \phi$ associated to a  of a critical peak in a global conflict reduces, by the convergence, to a reduced global conflict $\phi$. If this latter admits the following confluence diagram:

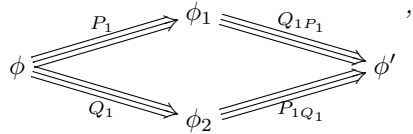
\centerline{
$$\xymatrix@R=.5em{
& &\D_1 \ar@3[drr]^{Q_{1P_1}} \\
\D \ar@3[drr]_{Q_1}\ar@3[urr]^{P_1}& & && \D'\\
& &\D_2 \ar@3[urr]_{P_{1Q_1}}
},$$}

then the confluence diagram of $ \phi$ is the following:

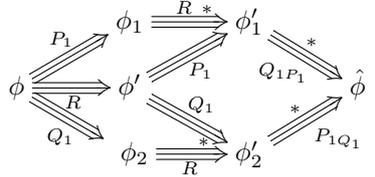
\centerline{
$$\xymatrix@R=.8em@C=1em{
&& \D_1 \ar@3[rr]^{~R~*}& &\D_1' \ar@3[drr]^{*}_{Q_{1P_1}} \\
 \D \ar@3[rr]_R \ar@3[drr]_{Q_1}\ar@3[urr]^{P_1} && \D' \ar@3[drr]^{Q_1}\ar@3[urr]_{P_1}& & && \hat \D\\
&&\ \D_2 \ar@3[rr]_R^{~~~*}& &\D_2' \ar@3[urr]^{*}_{~~P_{1Q_1}}
}.$$}
\end{proof}
\end{lemma}

\begin{teor}\label{simconv}

The rewriting system $\Sim$ is convergent.

\begin{proof} 
The rewriting system is terminating after Proposition \ref{termsim}, then  it suffices to prove the local confluence in order to have convergence. After Proposition \ref{simglob} and Lemma \ref{reducedglob}, by Lemma \ref{glob}, the local convergence is proved verifying the following $68$ critical peaks. We will handle them in batches: 
\begin{itemize}
\item $5$  \emph{coherence} peaks:
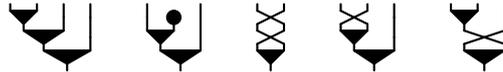
\MLine{\begin{gathered}
\twocell{(m *0 2) *1 (m*0 1) *1 m}						\qquad
\twocell{(1 *0 e *0 1) *1 (m *0 1)*1 m} 					\qquad
\twocell{s *1 s *1 m}								\qquad
\twocell{(s *0 1) *1  (m *0 1) *1 m}						\qquad
\twocell{(m *0 1) *1 s *1 m}							\qquad
\end{gathered}}
each conflict is generated by two non-structural rules and the confluence of diagrams is given by non-structural rules. Their confluence diagrams correspond to the coherence conditions given in Definitions \ref{MC} and \ref{SMC2};

\item $5$  \emph{Kelly}-peaks:
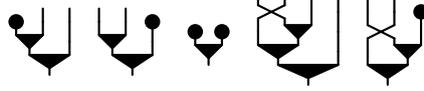
\MLine{
\begin{gathered}
\twocell{(e *0 2) *1 (m *0 1)*1 m} 						\quad
\twocell{(2 *0 e) *1 (m *0 1)*1 m} 						\quad
\twocell{(e *0 e) *1 m}								\quad
\twocell{(s *0 2) *1 (1 *0 m *0 1) *1 (m *0 1) *1 m} 				\quad
\twocell{(2 *0 e) *1 (s *0 1)*1 (1 *0 m)*1 m} 					
\end{gathered}
}
each conflict is generated by two non-structural rules and the confluence of diagrams is given by non-structural rules;

\item $12$  \emph{weak-Kelly}-peaks:
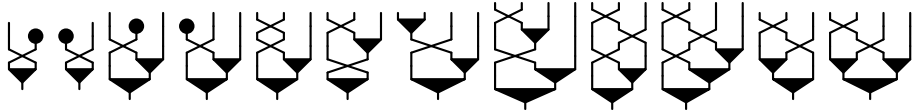
\MLine{
\begin{gathered}
\twocell{(1*0 e) *1 s *1 m}							\; \;
\twocell{(e*0 1) *1 s *1 m}							\; \;
\twocell{(1 *0 e *0 1) *1 (s *0 1)*1 (1 *0 m)*1 m} 				\; \;
\twocell{(e *0 2) *1 (s *0 1)*1 (1 *0 m)*1 m} 					\; \;
\twocell{(s *0 1) *1 (s *0 1) *1 (1 *0 m) *1 m} 				\; \; 
\twocell{(s *0 1) *1 (1 *0 m) *1 (s) *1 m } 					\; \;
\twocell{(m *0 2) *1 (s *0 1) *1 (1 *0 m) *1 m} 				\; \; 
\twocell{(s *0 2) *1 (1 *0 m *0 1) *1 (s *0 1) *1 (1 *0 m) *1 m} 		\; \; 
\twocell{(s *0 1) *1 (1*0 s) *1 (s *0 1) *1 (1 *0 m) *1 m}	 		\; \; 
\twocell{(s *0 2) *1 (1*0 s *0 1) *1 (s *0 m) *1 (1 *0 m) *1 m}	 	\; \; 
\twocell{(s *0 1) *1 (1 *0 s) *1 (m *0 1) *1 m}					\; \; 
\twocell{(s *0 2) *1 (1 *0 s *0 1) *1 (m *0 m) *1 m}				 
\end{gathered}
}
each conflict is generated by a non-structural rule and a structural rule. This latter is the only structural rule in the confluence diagram;

\item $18$ \emph{simply-foldable} peaks: 
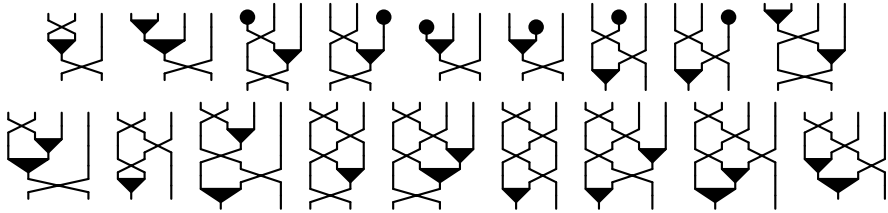
\MLine{
\begin{gathered}
\twocell{(s*0 1) *1 (m *0 1) *1 s} 							\quad
\twocell{(m*0 2) *1 (m *0 1) *1 s} 							\quad
\twocell{(e *0 2) *1 (s *0 1)*1 (1 *0 m)*1 s}	 					\quad
\twocell{(2 *0 e) *1 (s *0 1)*1 (1 *0 m)*1 s}	 					\quad
\twocell{(e *0 2) *1 (m *0 1) *1 s}	 						\quad
\twocell{(1 *0 e *0 1) *1 (m *0 1) *1 s}	 						\quad
\twocell{(1 *0 e *0 1) *1 (s *0 1)*1 (1 *0 s)*1 (m *0 1)} 				\quad
\twocell{(2 *0 e) *1 (s *0 1)*1 (1 *0 s)*1 (m *0 1)} 					\quad
\twocell{(m *0 2) *1 (s *0 1)*1 (1 *0 m)*1 s} 						\\
\twocell{(s *0 2) *1 (1 *0 m *0 1) *1 (m *0 1) *1 s}					\quad
\twocell{(s *0 1) *1 (1 *0 s) *1 (s *0 1) *1 (m *0 1)  } 				\quad
\twocell{(s *0 2) *1 (1 *0 m *0 1) *1 (s *0 1) *1 (1 *0 s) *1(m *0 1)} 		\quad
\twocell{(s *0 1) *1 (1*0 s) *1 (s *0 1)*1 (1 *0 m)*1 s} 				\quad
\twocell{(s *0 2) *1 (1*0 s *0 1) *1 (s *0 m)*1 (1 *0 m ) *1 (s)}			\quad
\twocell{(s *0 1) *1 (1*0 s) *1 (s *0 1) *1 (1 *0 s) *1 (m *0 1)} 			\quad
\twocell{(s *0 2) *1 (1*0 s *0 1) *1 (s *0 m)*1 (1 *0 s) *1 (m *0 1)}		\quad
\twocell{(s *0 2) *1 (1*0 s *0 1) *1 (s *0 s)*1 (1 *0 m *0 1) *1 (m *0 1)}		\quad
\twocell{(s *0 2) *1 (1*0 s *0 1) *1 (m *0 s)*1 (m *0 1)}				
\end{gathered}
}
the confluence diagrams are made of parallel $3$-cells composed of the one occurrence of same non-structural rule and some structural ones;

\item $28$ \emph{strongly-foldable} peaks:
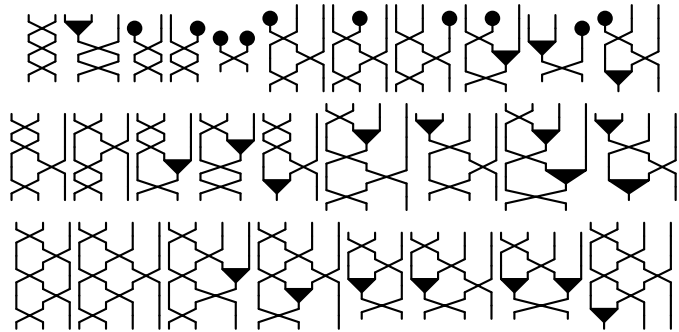
\MLine{
\begin{gathered}
\twocell{s *1 s *1 s}								\;
\twocell{(m*0 1) *1 s *1 s}							\;
\twocell{(e*0 1) *1 s *1 s}							\;
\twocell{(1*0 e) *1 s *1 s}							\;
\twocell{(e *0 e) *1 s}								\;
\twocell{(e *0 2) *1 (s *0 1)*1 (1 *0 s)*1 (s *0 1)} 				\;
\twocell{(1 *0 e *0 1) *1 (s *0 1)*1 (1 *0 s)*1 (s *0 1)} 			\;
\twocell{(2 *0 e) *1 (s *0 1)*1 (1 *0 s)*1 (s *0 1)} 				\;
\twocell{(1 *0 e *0 1) *1 (s *0 1)*1 (1 *0 m)*1 s}			 	\;
\twocell{(2 *0 e) *1 (m *0 1) *1 s}	 					\;
\twocell{(e *0 2) *1 (s *0 1)*1 (1 *0 s)*1 (m *0 1)} 				\\
\twocell{(s *0 1) *1 (s *0 1)*1 (1 *0 s)*1 (s *0 1)} 				\;
\twocell{(s *0 1)*1 (1 *0 s)*1 (s *0 1) *1 (s *0 1)} 				\;
\twocell{(s *0 1) *1 (s *0 1)*1 (1 *0 m)*1 s} 					\;
\twocell{(s *0 1)*1 (1 *0 m)*1 s *1 s} 						\;
\twocell{(s *0 1) *1 (s *0 1) *1 (1 *0 s) *1 (m *0 1)}				\;
\twocell{(s *0 2) *1 (1 *0 m *0 1) *1 (s *0 1)*1 (1 *0 s)*1 (s *0 1)}	\;
\twocell{(m*0 2)*1 (s *0 1)*1 (1 *0 s) *1 (s *0 1)}				\;
\twocell{(s *0 2) *1 (1 *0 m *0 1) *1 (s *0 1)*1 (1 *0 m)*1 s}		\;
\twocell{(m *0 2) *1 (s *0 1) *1 (1 *0 s) *1 (m *0 1) } 			\\
\twocell{(s *0 1) *1 (1*0 s) *1 (s *0 1)*1 (1 *0 s)*1 (s *0 1)} 		\;
\twocell{(s *0 2) *1 (1*0 s *0 1) *1 (s *0 s)*1 (1 *0 s *0 1) *1 (s *0 2)}	\;
\twocell{(s *0 2) *1 (1*0 s *0 1) *1 (s *0 m)*1 (1 *0 s)*1 (s *0 1)} 	\;
\twocell{(s *0 2) *1 (1*0 s *0 1) *1 (s *0 s)*1 (1 *0 m *0 1) *1 (s *0 1)}	\;
\twocell{(s *0 1) *1 (1*0 s) *1 (m *0 1) *1 s}					\;
\twocell{(s *0 2) *1 (1*0 s *0 1) *1 (m *0 s)*1 (s *0 1)}			\;
\twocell{(s *0 2) *1 (1*0 s *0 1) *1 (m *0 m)*1 s}				\;
\twocell{(s *0 2) *1 (1*0 s *0 1) *1 (s *0 s)*1 (1 *0 s *0 1) *1 (m *0 2)}	
\end{gathered}
}
the $3$-cells in confluence diagrams are made of structural rules: 
\end{itemize}

The confluence diagrams for these critical peaks are discussed in Appendix \ref{appConf}.
\end{proof}

\end{teor}

\begin{cor}\label{monconv}
The rewriting systems $\M$ is convergent.
\begin{proof} 
Since  $\M$ terminates (Corollary \ref{monter}), the convergence of $\M$ can be evinced by Proposition \ref{simconv}. In fact, the solutions of its critical peaks 
$$
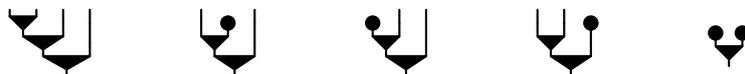
\twocell{(( m *0 2) *1 m *0 1) *1 m}  \qquad \qquad
\twocell{((1 *0 e *0  1)*1 (m *0 1)) *1 m} \qquad \qquad
\twocell{((e *0 2)*1 (m *0 1)) *1 m} \qquad \qquad
\twocell{((2 *0 e)*1 (m *0 1)) *1 m} \qquad \qquad
\twocell{(e *0 e)*1 m}
$$
are made only of $3$-cells in $\M$ (See Appendix \ref{appconf}).
\end{proof}
\end{cor}

\begin{cor}\label{structconf}
If $\phi$ is a diagram in $\Sim$, then there is a unique $\bar \phi$ in $\Sim$ such that $\xymatrix@C=.8em{\phi \ar@3^{*}[r] & \bar \phi}$ and no structural rule can be applied to $\bar \phi$.
\end{cor}

\section{Coherence of symmetric monoidal categories}

In this section, we give a proof of coherence theorems using rewriting. Since the confluence of our rewriting systems which present functors and natural transformations of these theories are proved, we are able to define a $4$-cell for any pair of parallel $3$-cell once a sufficient large set of $4$-cells is defined. 
Finally we are able to prove the theorems establishing a correspondence between $4$-cell and commutative or trivial diagrams.

\begin{defi}
We extend the rewriting system $\Sim$ to a $(4,3)$-polygraph $\bar\Sim$ by adding the following (unoriented) $4$-cells:

$$
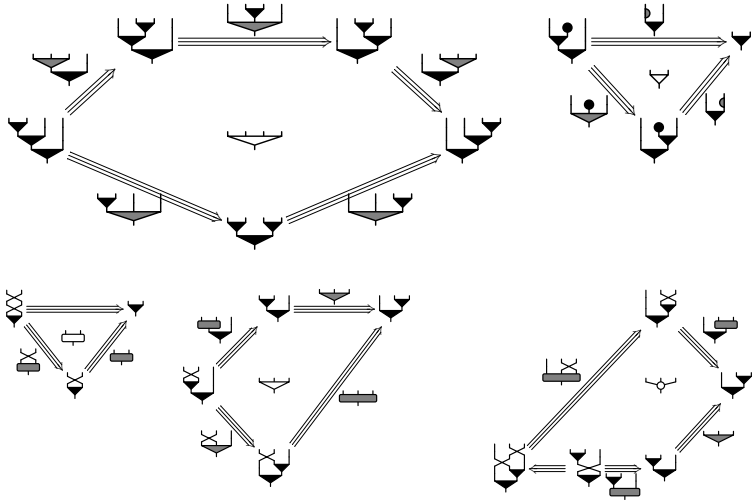
\resizebox{10cm}{!}{
\xymatrix{
& {\twocell{(( 1 *0 m *0 1) *1 m *0 1) *1 m}} \ar@3[rr]^{\twocell{( 1 *0 m *0 1) *1 alpha }}&& {\twocell{((1 *0 m *0 1) *1 1 *0 m) *1 m}} \ar@3[dr]^{\twocell{(1*0 alpha)*1 m}} \\
{\twocell{(( m *0 2) *1 m *0 1) *1 m} } \ar@3[ur]^{\twocell{(alpha *0 1) *1 m}} \ar@3[drr]_{\twocell{(m*0 2) *1 alpha}}&&{\twocell{penta}}&& {\twocell{(( 2 *0 m ) *1 (1 *0 m)) *1 m}} \\
& &{\twocell{( m *0 m)  *1 m }}\ar@3[urr]_{\twocell{(2*0 m)*1 alpha}}
}
\qquad
\xymatrix{
{\twocell{((1 *0 e *0  1)*1 (m *0 1)) *1 m}} \ar@3[rr]^{\twocell{(r *0 1) *1 m}}  \ar@3[dr]_{\twocell{(1 *0 e *0 1) *1 alpha}} & & {\twocell{m}} \\
& {\twocell{((1 *0 e *0  1)*1 (1 *0 m))*1 m}} \ar@3[ur]^{\twocell{tria}  ~~~~}_{\twocell{(1*0 l)*1 m}}
}
}
$$
$$
\resizebox{10cm}{!}{
\xymatrix{
\twocell{s *1 s *1 m} \ar@3[rr] \ar@3[dr]_{\twocell{ s *1 tau}} & & {\twocell{m}} \\
& \twocell{s *1 m} \ar@3[ur]_{\twocell{tau}}^{\twocell{inv}  ~~~}
}
\qquad
\xymatrix{
& \twocell{( m *0 1) *1 m} \ar@3[rr]^{\twocell{alpha}}& & \twocell{(1 *0 m)*1 m}  &
\\
\twocell{(s *0 1) *1  (m *0 1) *1 m} \ar@3[dr]_{\twocell{(s *0 1) *1 alpha}}  \ar@3[ur]^{\twocell{(tau *0 1) *1 m}} & \twocell{g}
\\
&\twocell{(s *0 1) *1 ( 1 *0 m) *1 m}\ar@3[uurr]_{\twocell{gamma}} 
}
\qquad
\xymatrix{
&&  \twocell{(1*0 s) *1 (1 *0 m)*1 m}\ar@3[dr]^{\twocell{(1*0 tau) *1 m}} 
\\
&  &\twocell{exa2}& \twocell{(1 *0 m)*1 m}
\\
  \twocell{(1 *0 s) *1 (s *0 1) *1  (1 *0 m) *1 m}  \ar@3[uurr]^{\twocell {(1*0 s)*1 gamma}} &\twocell{(m*0 1) *1 s *1 m}\ar@3[l]\ar@3[r]_{\twocell{(m *0 1) *1 tau}}& \twocell{( m *0 1) *1 m} \ar@3[ru]_{\twocell{alpha}}
}
%\xymatrix{
%&   & &  \twocell{(1*0 s) *1 (1 *0 m)*1 m}\ar@3[dr]^{\twocell{(1*0 tau) *1 m}} \\
%&& &\twocell{exa2}& \twocell{(1 *0 m)*1 m}\\
%&  \twocell{(1 *0 s) *1 (s *0 1) *1  (1 *0 m) *1 m}  \ar@3[uurr]^{\twocell {(1*0 s)*1 gamma}} & \twocell{(m*0 1) *1 s *1 m}\ar@3[l]\ar@3[r]_{\twocell{(m *0 1) *1 tau}}& \twocell{( m *0 1) *1 m} \ar@3[ru]_{\twocell{alpha}}
%}
}
$$
plus the set of $4$-cells with border a solution of a simply-foldable or strongly-foldable critical peak (we denote them with the symbol $\comm$).

Similarly $\bar \M$ is the  $(4,3)$-polygraph  obtained adding to $\M$ the $4$-cells $\twocell {penta}$ and $\twocell{tria}$.
\end{defi}

\begin{prop}[Cellular Kelly's lemma for SMC]\label{kellysim} \hfill \\ 
For all \emph{Kelly-peaks} and \emph{weak-Kelly-peaks}, a $4$-cell can be defined from set of $4$-cells in $\bar \Sim$.
\begin{proof}
See Appendix \ref{CellyKelly}
\end{proof}
\end{prop}

\begin{cor}[Cellular version of Kelly's lemma]\label{Kelly}
The $4$-cells \twocell{penta} and \twocell{tria} suffice to define a $4$-cell for every critical peak of $\M$.
\end{cor}

%\begin{teor}[Coherence theorem for monoidal categories, Mac Lane \cite{Mac}]\label{cohemc}\hfill \\
%Every diagrams made of $\alpha, \lambda, \rho$  in an arbitrary monoidal category $\C$ commutes.
%\end{teor}
%
%
%
%
%\begin{proof}  The rewriting system $\M$ has exactly one $3$-cell for each natural transformation in a monoidal category, which guarantees a one-to-one correspondence with pairs of parallel $3$-cells with target a $2$-cell with one output. The two coherence conditions defined by the triangular and the pentagonal identities corresponds to the $4$-cells \twocell{penta} and \twocell{tria} and, by Cor.\ref{Kelly}, they suffice to generate all $4$-cells with border any critical peak. Moreover, the convergence of $\M$ (Prop.\ref{monconv}) allows one to find (using Lemma \ref{pav}) a $4$-cell with border any pair of parallel $3$-cells just using  the aforementioned $4$-cells. 
%
%Finally, the commutativity of each diagram in the monoidal category theory follows by the correspondence between confluence diagrams and parallel $3$-cells and, in particular, between commutative diagrams and $4$-cells.
%\end{proof}

We are now able to prove the coherence theorems.

\begin{proof}[Proof of Theorem \ref{cohemc}]
In order to prove the commutativity of all confluence diagrams in a monoidal category theory, we use the correspondence between cells of $\bar \M$ and functors, natural transformations and coherence conditions of this theory.

We have a one-to-one correspondence between the generating $2$-cells $\twocell m$ and $\twocell e$ and the functors $\alpha$ and $\e$ which allows us to establish a one-to-one correspondence between  the $2$-cells in $\bar \M$ and  the objects of a monoidal category.

Moreover, there two more one-to-one correspondences: one between $3$-cells $\twocell {alpha}, \twocell {r}$ and $\twocell {l}$ and the (families of) natural transformations $\alpha, \rho, \lambda$ and another between the $4$-cells $\twocell{penta}$ and $\twocell {tria}$ and the coherence conditions.

By this fact, whenever we have a diagram in a symmetric monoidal category made of $\alpha, \rho$ and $ \lambda$, we can associate it a unique confluence diagram in $\M$ obtained by replacing the objects with the corresponding $2$-cells, and the arrows with the corresponding $3$-cells.

Since $\M$ is confluent (Corollary \ref{monconv}),  Lemma \ref{pav} assures the possibility to define a composed $4$-cell in $\bar \M$ using only the $4$-cells associated to its critical peaks solutions. To conclude, we remind that the cellular version of Kelly's lemma (Corollary \ref{Kelly}) asserts that the $4$-cells $\twocell{penta}$ and $\twocell {tria}$ suffice for this scope.
\end{proof}

\begin{proof}[Proof of Theorem \ref{cohesmc}]\label{cohesmcproof}
As in the previous proof, we wish to establish a correspondence between the cells in $\bar \Sim$ and functors, natural transformations and coherence conditions in a symmetric monoidal category.

We immediately observe an incongruence between these two syntaxes. This is due to the fact that the $2$-cell \twocell{s} does not correspond to any functor in the symmetric monoidal category theory. Its introduction is due to the fact that, in general, in order to represent an \emph{algebraic theory} through diagrams \cite{LafEq}, we need to define some operators to manage resources: \emph{duplication}, \emph{erasing} and \emph{permutation} (in this particular case, the fact that we only consider linear diagrams, rules out the necessity of duplication nor erasing).

We give a \emph{standard interpretation} of \twocell{s}, in the sense that we consider an equality between the pairs $(x,y)$ and $(y,x)$. This leads to the identification of  some $2$-cells interpretation which corresponds to the same term and, consequently, the trivial interpretation (by an identity) of any $3$-cells of a structural rule (Definition \ref{SimDef}.\ref{manr}). 

Moreover, this identification induces a trivial interpretation of the $4$-cells in $\bar\Sim$ associate to the foldable peaks. In fact, once we respectively interpret the $2$-cells and $3$-cells as object and arrows in the symmetric monoidal category theory, the interpretation of these $4$-cells collapses to a trivial diagram in the theory, composed of an arrow (simply-foldable peaks) or a single object (strongly-foldable peaks). In fact, this behavior justifies the choice made in Proposition \ref{kellysim} of a unique notation $\comm$ for all these $4$-cells.

However, these two syntaxes show a one-to-one correspondence between the (families of) natural transformations $\alpha, \rho, \lambda, \tau$ and $\gamma$ and the non-structural rewriting rules $\begin{Bmatrix}\twocell {alpha}, \twocell {r}, \twocell {l}, \twocell {tau},\twocell {gamma}\end{Bmatrix}$, as well as a one-to-one correspondence between coherence conditions and the set of  $4$-cells $\begin{Bmatrix}\twocell{penta}, \twocell{tria}, \twocell{g}, \twocell{inv}, \twocell{exa2}\end{Bmatrix}$.

In order to prove the coherence, we build a confluence digram $D_\Sim$ in $\Sim$ associate to a diagram $D$ in a symmetric monoidal category the following way:
\begin{itemize}
\item We fix an order over the variable occurring in the terminal object of $D$;
\item For each object $o$ in $D$, we consider a $2$-cell $\phi_{o}$ which can be interpreted as $o$ whenever the input strings labels is the one fixed in the first step. Moreover we chose the unique (after Corollary \ref{structconf}) $2$-cell $\phi_o$ such that no structural rewriting rule can be applied to it. We call such $2$-cells \emph{primary};
\item For each morphism $o\fr_x o'$ in $D$ we add to $D_\Sim$ the corresponding non-structural \emph{primary} $3$-cell $\rew {\phi_o} {\phi_{o'}}$ between the two related $2$-cells. Since the source is primary, also its target is; then this $2$-cell already is in $D_\Sim$. If the source of such $3$-cell is not in $D_\Sim$, we add the related $2$-cell to $D_\Sim$ -- this new $2$-cell is not primary but its interpretation is an object of the diagram $D$;
\item For each of this new $2$-cell $\phi_{o}'$, we add to $D_\Sim$ the $2$-cells and the $3$-cells of the rewriting path from $\phi_o'$ to the natural $2$-cell $\phi_o$ which has the same interpretation.
\end{itemize}

This new diagram $D_\Sim$ can be decomposed by means of confluence diagrams in $\Sim$ with some $2$ and/or $3$-cells in common.  Indeed, it suffices to solve its conflicts, which are the pair of $3$-cells with a common source. If the source of the conflict is a primary $2$-cell, then conflict is between two primary $3$-cells. If the source is not primary, then the conflict is between the primary $3$-cell which introduced this $2$-cell and the rewriting path from this latter to its corresponding primary $2$-cell.

The confluence of $\Sim$ (Propositions \ref{simconv}) guarantees the decomposition of $D_\Sim$ by means of confluence diagrams which, by Lemma \ref{pav}, in turn can be decomposed by using the $4$-cells in $\bar \Sim$ associated to its critical pairs. Moreover, Proposition \ref{kellysim} guarantees that we only need the $4$-cells $\begin{Bmatrix}\twocell{penta}, \twocell{tria}, \twocell{g}, \twocell{inv}, \twocell{exa2}\end{Bmatrix}$ of the coherence peaks together with the $4$-cells of the foldable peaks (labeled by $\comm$).

We conclude interpreting the $2$,$3$ and $4$-cells in the given $D_\Sim$ decomposition as respectively objects, morphisms and coherence conditions.
\end{proof}

By means of example, let us construct the decomposition of the hexagonal identity in Definition \ref{SMC} by the two coherence conditions given to define the natural transformation $\gamma$ in Definition \ref{SMC2}: in the first stage we have the diagram with only the primary $2$-cells, in the second stage we add the $3$-cells, the non-primary $2$-cells and the relative rewriting path and in the third stage we have the decomposition of the diagram by means of $4$-cells.
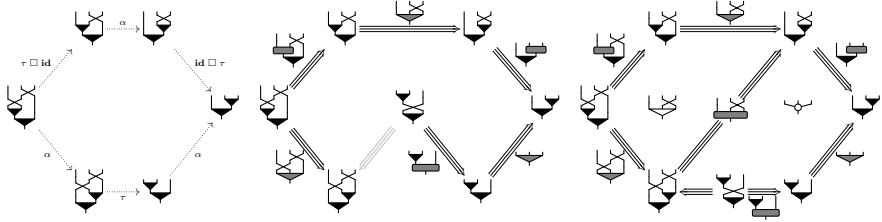
\begin{center}
\resizebox{.5\width}{!}{
$$
\xymatrix{
&   \twocell{(1*0 s) *1 ( m *0 1) *1 m} \ar@{.>}[r]^{\alpha}&
\twocell{(1*0 s) *1 (1 *0 m)*1 m} \ar@{.>}[dr]^{\id \pro \tau}\\
\twocell{(1 *0 s) *1 (s *0 1) *1  (m *0 1) *1 m} \ar@{.>}[dr]_{\alpha}\ar@{.>}[ur]^{\tau \pro \id}& 
 && 
\twocell{(1 *0 m)*1 m}\\
& 
\twocell{(1 *0 s) *1 (s *0 1) *1  (1 *0 m) *1 m} \ar@{.>}[r]_{\tau}& 
\twocell{( m *0 1) *1 m} \ar@{.>}[ur]_{\alpha}
}
\quad
\xymatrix{
&   \twocell{(1*0 s) *1 ( m *0 1) *1 m} \ar@3[rr]^{\twocell{(1*0 s) *1 alpha}}&& 
\twocell{(1*0 s) *1 (1 *0 m)*1 m}\ar@3[dr]^{\twocell{(1*0 tau) *1 m}} \\
\twocell{(1 *0 s) *1 (s *0 1) *1  (m *0 1) *1 m} \ar@3[dr]_{\twocell{(1 *0 s) *1 (s *0 1) *1  alpha}} \ar@3[ur]^{\twocell{(1*0 s) *1 (tau *0 1) *1 m}} & 
 &\twocell{(m*0 1) *1 s *1  m}\ar@3@[lightgray][dl]\ar@3[dr]_{\twocell{(m *0 1) *1 tau}} &  &
\twocell{(1 *0 m)*1 m}\\
& 
\twocell{(1 *0 s) *1 (s *0 1) *1  (1 *0 m) *1 m} & & 
\twocell{( m *0 1) *1 m} \ar@3[ru]_{\twocell{alpha}}
}
\quad
\xymatrix{
&   \twocell{(1*0 s) *1 ( m *0 1) *1 m} \ar@3[rr]^{\twocell{(1*0 s) *1 alpha}}&& 
\twocell{(1*0 s) *1 (1 *0 m)*1 m}\ar@3[dr]^{\twocell{(1*0 tau) *1 m}} \\
\twocell{(1 *0 s) *1 (s *0 1) *1  (m *0 1) *1 m} \ar@3[dr]_{\twocell{(1 *0 s) *1 (s *0 1) *1  alpha}} \ar@3[ur]^{\twocell{(1*0 s) *1 (tau *0 1) *1 m}} & 
\twocell{(1 *0 s) *1 g} && \twocell{exa2} &
\twocell{(1 *0 m)*1 m}\\
& 
\twocell{(1 *0 s) *1 (s *0 1) *1  (1 *0 m) *1 m} \ar@3[uurr]|-{\twocell {(1*0 s)*1 gamma}}& \twocell{(m*0 1) *1 s *1  m}\ar@3[l]\ar@3[r]_{\twocell{(m *0 1) *1 tau}}& 
\twocell{( m *0 1) *1 m} \ar@3[ru]_{\twocell{alpha}}
}
$$
}
\end{center}
This allows us to pass from the diagram corresponding to the hexagonal identity to its decomposition by means of coherence conditions of Definition \ref{SMC2}:

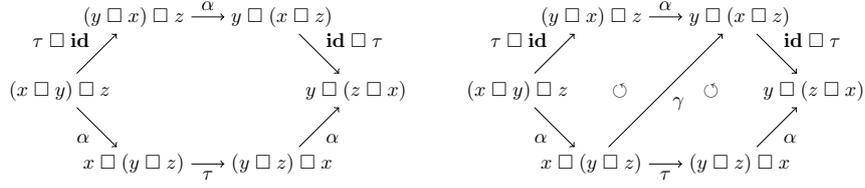
\resizebox{.75\width}{!}{
\begin{tikzpicture}[node distance=1.3 cm, auto]
  \node (0) {$(x \pro y)\pro z$};
  \node (1) [right of=0, below of =0] {$x\pro (y \pro z)$};
\node(c)[right of=1, above of =1] {};
  \node (2) [below of=c, right of =c] {$(y \pro z) \pro x$};
  \node(3)[right of=2, above of =2]{$y\pro (z\pro x)$};
  \node (4) [right of=0, above of =0] {$(y \pro x) \pro z$};
  \node (5) [above of=c, right of =c] {$y \pro (x \pro z)$};

\draw [->] (0) to node[swap]{$\alpha$} (1);
\draw [->] (1) to node[swap]{$\sym$} (2);
\draw [->] (2) to node[swap] {$\alpha$}(3);

\draw [->] (0) to node{$\sym \pro \id$} (4);
\draw[->](4) to node {$\alpha$}(5);
\draw[->](5) to node {$\id \pro \sym$}(3);
\end{tikzpicture}
\qquad
\begin{tikzpicture}[node distance=1.3 cm, auto]
  \node (0) {$(x \pro y)\pro z$};
  \node (1) [right of=0, below of =0] {$x\pro (y \pro z)$};
\node(c)[right of=1, above of =1] {};
  \node (2) [below of=c, right of =c] {$(y \pro z) \pro x$};
  \node(3)[right of=2, above of =2]{$y\pro (z\pro x)$};
  \node (4) [right of=0, above of =0] {$(y \pro x) \pro z$};
  \node (5) [above of=c, right of =c] {$y \pro (x \pro z)$};

\node(c1)[right of=c, node distance =.8cm]{$\circlearrowleft$};
\node(c2)[left of=c, node distance =.8cm]{$\circlearrowleft$};

\draw [->] (0) to node[swap]{$\alpha$} (1);
\draw [->] (1) to node[swap]{$\sym$} (2);
\draw [->] (2) to node[swap] {$\alpha$}(3);

\draw [->] (0) to node{$\sym \pro \id$} (4);
\draw[->](4) to node {$\alpha$}(5);
\draw[->](5) to node {$\id \pro \sym$}(3);

\draw [->] (1) to node[swap] {$\gamma$}(5);
\end{tikzpicture}

}

\begin{cor}[Kelly's lemma for SMC]
The five coherence conditions given in Definitions \ref{MC} and \ref{SMC2} suffice to prove coherence for all critical branching between natural transformations in any symmetric monoidal category.
\begin{proof}
The proof follows from the Proposition \ref{kellysim} and the interpretation of the $4$-cells according to the construction given in Proof \ref{cohesmcproof}.
\end{proof}
\end{cor}

\section{Conclusions}
In this paper we have presented the detailed proof of confluence of Lafon's rewriting systems for symmetric monoidal categories proposed in \cite{LafBool} completing it by exhibiting the confluence diagrams of the critical peaks of the rewriting system  $\Sim$. This allows us to apply some tools of polygraph theory (\cite{GuirMalHD}, \cite{GuirMalCohe}, \cite{GuirMalHom} and \cite{maxime}) to adapt this proof into a proof of coherence for symmetric monoidal categories. In fact, we are able to give a new constructive procedure to prove the coherence theorem in the setting of a $(4,3)$-polygraph. Differently from the proofs  given in \cite{GuirMalCohe} and  \cite{maxime}, where a $(4,2)$-polygraph setting is required in order to consider the structural rules in Definition \ref{SimDef}.\ref{manr} as equalities, in our proof we can keep structural rules oriented even if their interpretations are identities.

Moreover, in this paper we work with an alternative presentation of the symmetric monoidal category theory. This is given replacing the hexagonal identity by two coherence conditions induced by the decomposition of this diagram by means of the so-called \emph{parallel associativity}.
Thus, we need to prove an analogous of Kelly's Theorem \cite{Kelly} in order to show that the five initial coherence conditions suffice to prove the coherence for all confluence diagrams of critical branching between natural transformations. In fact, in monoidal category theory Kelly shows the sufficiency of the pentagonal and triangular identity to prove the coherence, while the original  Mac Lane's \cite{Mac} definition gives five coherence conditions corresponding to all critical branchings between natural transformations.

For this purpose, in Appendix \ref{appConf} we show the confluence diagrams of the critical pairs of Lafont's rewriting system $\Sim$ completing the proof given in \cite{LafBool}. These confluence diagrams can be interpreted as  confluence diagrams made of morphisms of symmetric monoidal category theory.

Furthermore, in Appendix \ref{CellyKelly} we give some possible decompositions of the critical peaks of $\Sim$ which have non-trivial interpretation in the theory. These decompositions are made by means of some confluence diagrams whose interpretations correspond to the coherence conditions of our symmetric monoidal category theory presentation.

%%%%%%%%%%%%%%%%%%%%%%%%%%%%%%%%

%Since $\Sim$ is confluent (Propositions \ref{simconv}), once given a set of $4$-cells with border the solutions of all critical peaks, one can find a $4$-cell with border each pair of parallel $3$-cells by lemma \ref{pav}. Moreover, by Proposition \ref{kellysim} and Lemma \ref{exagamma} this set of $4$-cells can be restrict to the set
%$$\begin{Bmatrix}\twocell{penta}, \twocell{tria}, \twocell{g}, \twocell{inv}, \twocell{exa}\end{Bmatrix}$$ 
%plus the set of $4$-cells we note with $\comm$, i.e. with border a solution of a trivial or strongly-trivial critical peak.
%These cells correspond respectively to pentagonal, triangular, hexagonal, $\gamma$-defining and $\sym$-involutivity coherence conditions plus some trivial coherence conditions. 
%
%
%The commutativity of each linear diagram of symmetric monoidal category theory follows by the correspondence between $4$-cells and commutative (or trivial) diagrams. In this case some linear diagram can not be represented by a pair of parallel $3$-cells (for example the hexagonal identity), but we are still able to give a $4$-cell with border a pair of unoriented $3$-cell which modulo standard interpretation of \twocell{s} represent the diagram.

%%
%% Bibliography
%%

\appendix

%\end{document}

\newpage

\section{Critical peaks confluence diagrams}\label{appConf}

\begin{itemize}
\item The $5$ \emph{coherence} peaks

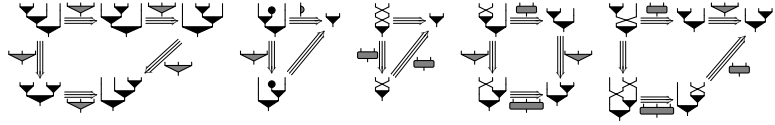
\resizebox{.5\width }{!}{
$$ \xymatrix {
{\twocell{(( m *0 2) *1 m *0 1) *1 m} } \ar@3_{\twocell{alpha}}[d]  \ar@3^{\twocell{alpha}}[r] & 
{\twocell{(( 1 *0 m *0 1) *1 m *0 1) *1 m}} \ar@3^{\twocell{alpha}}[r]&
{\twocell{((1 *0 m *0 1) *1 1 *0 m) *1 m}} \ar@3^{\twocell{alpha}}[dl] \\
{\twocell{( m *0 m)  *1 m }}\ar@3_{\twocell{alpha}}[r]&
{\twocell{(( 2 *0 m ) *1 (1 *0 m)) *1 m}}
}
$$
$$
\xymatrix{
{\twocell{((1 *0 e *0  1)*1 (m *0 1)) *1 m}} \ar@3^{\twocell{r}}[r] \ar@3_{\twocell{alpha}}[d]& {\twocell{m}} \\
{\twocell{((1 *0 e *0  1)*1 (1 *0 m))*1 m}} \ar@3[ur]
}
$$
$$
\xymatrix{
\twocell{s *1 s *1 m} \ar@3[r] \ar@3_{\twocell{tau}}[d]& {\twocell{m}} \\
\twocell{s *1 m} \ar@3_{\twocell{tau}}[ur]
}$$
$$
\xymatrix{
\twocell{(s *0 1) *1  (m *0 1) *1 m} \ar@3^{\twocell{tau}}[r] \ar@3_{\twocell{alpha}}[d]&
\twocell{( m *0 1) *1 m} \ar@3[d]^{\twocell{alpha}}\\
\twocell{(s *0 1) *1 ( 1 *0 m) *1 m}\ar@3_{\twocell{gamma}}[r] &\twocell{(1 *0 m)*1 m}
}$$
$$
\xymatrix{
 \twocell{(m*0 1) *1 s *1 m}\ar@3[r]^{\twocell{tau}}\ar@3[d]&
\twocell{( m *0 1) *1 m} \ar@3^{\twocell{alpha}}[r]& 
\twocell{(1 *0 m)*1 m}\\
\twocell{(1 *0 s) *1 (s *0 1) *1  (1 *0 m) *1 m} \ar@3[r]_{\twocell{gamma}}& 
\twocell{(1*0 s) *1 (1 *0 m)*1 m}\ar@3_{\twocell{tau}}[ru] 
}$$
}

\item The $5$ \emph{Kelly}-peaks:

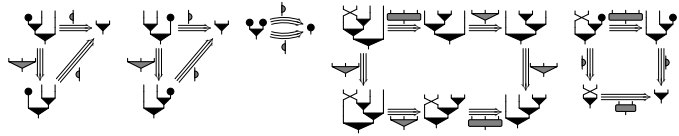
\resizebox{.5\width}{!}{
$$
\xymatrix{
{\twocell{((e *0  2)*1 (m *0 1)) *1 m}} \ar@3^{\twocell{l}}[r] \ar@3[d]_{\twocell{alpha}}& {\twocell{m}} \\
{\twocell{((e *0  2)*1 (1 *0 m))*1 m}} \ar@3_{\twocell{l}}[ur]
}$$
$$
\xymatrix{
{\twocell{((2 *0 e)*1 (m *0 1)) *1 m}} \ar@3^{\twocell{r}}[r] \ar@3_{\twocell{alpha}}[d]&{\twocell{m}} \\
{\twocell{((2 *0  e)*1 (1 *0 m))*1 m}} \ar@3_{\twocell{r}}[ur]
}$$
$$
\xymatrix{
\twocell{(e *0  e)*1 m} \ar@3@/_/_{\twocell{l}}[r] \ar@3@/^/^{\twocell{r}}[r] & {\twocell{e}}
}$$
$$
\xymatrix{
\twocell{(s *0 2) *1 (1 *0 m *0 1) *1 (m *0 1) *1 m} \ar@3^{\twocell{gamma}}[r]\ar@3_{\twocell{alpha}}[d]&
\twocell{(1 *0 m *0 1) *1 (m *0 1) *1 m} \ar@3^{\twocell{alpha}}[r]& 
\twocell{(1 *0 m *0 1) *1 (1 *0 m) *1 m}  \ar@3^{\twocell{alpha}}[d] \\
\twocell{(s *0 2) *1 (1 *0 m *0 1) *1 (1 *0 m) *1 m}   \ar@3_{\twocell{alpha}}[r] & 
\twocell{ (s *0 m ) *1 (1 *0 m) *1 m}\ar@3_{\twocell{gamma}}[r] &
\twocell{(2 *0 m) *1 (1 *0 m) *1 m} 
}$$
$$
\xymatrix{
{\twocell{(s *0 e)*1 (1 *0 m) *1 m}} \ar@3^{\twocell{gamma}}[r] \ar@3_{\twocell{r}}[d]& \ar@3^{\twocell{r}}[d]{\twocell{(2 *0 e) *1 (1 *0 m) *1 m}}\\
 \twocell{s *1 m} \ar@3_{\twocell{tau}}[r] & \twocell{m}
}
$$
}

\item The $12$  \emph{weak-Kelly}-peaks:

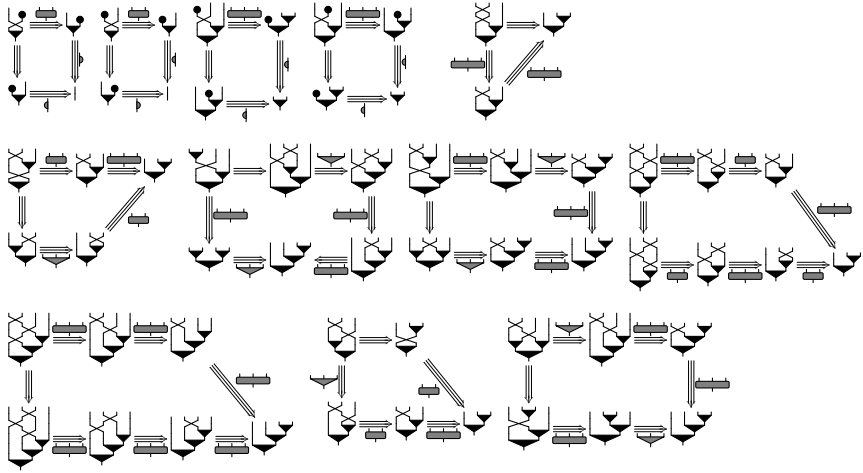
\resizebox{.5\width}{!}{
$$
\xymatrix{
{\twocell{(1 *0 e)*1 (s) *1 m}} \ar@3^{\twocell{tau}}[r] \ar@3[d]& \twocell{(1 *0 e) *1 m} \ar@3^{\twocell{r}}[d]\\
{\twocell{(e *0 1) *1 m}} \ar@3_{\twocell{l}}[r] & \twocell{1}
}$$
$$
\xymatrix{
{\twocell{(e *0 1)*1 (s) *1 m}} \ar@3^{\twocell{tau}}[r] \ar@3[d]& \twocell{(e *0 1) *1 m} \ar@3^{\twocell{l}}[d]\\
{\twocell{(1 *0 e) *1 m}} \ar@3_{\twocell{r}}[r] & \twocell{1}
}$$

$$
\xymatrix{
{\twocell{(e *0 2) *1 (s *0 1)*1 (1 *0 m) *1 m}} \ar@3^{\twocell{gamma}}[r] \ar@3[d]& {\twocell{(e *0  m) *1 m}} \ar@3^{\twocell{l}}[d]\\
{\twocell{(1*0 e *0 1) *1 (1 *0 m) *1 m}} \ar@3_{\twocell{l}}[r] & \twocell{m}
}$$
$$
\xymatrix{
{\twocell{(1*0 e *0 1)*1 (s *0 1)*1 (1 *0 m) *1 m}} \ar@3^{\twocell{gamma}}[r] \ar@3[d]& {\twocell{(1*0 e *0 1) *1 (1 *0 m) *1 m}} \ar@3^{\twocell{l}}[d]\\
 {\twocell{(e *0  m) *1 m}} \ar@3_{\twocell{l}}[r] & \twocell{m}
}\qquad
$$

$$\xymatrix{
\twocell{(s *0 1) *1 (s *0 1) *1 (1 *0 m) *1 m}  \ar@3[r] \ar@3[d]_{\twocell{gamma}} &  \twocell{(1*0 m) *1 m}\\
\twocell{(s *0 1) *1 (1 *0 m) *1 m} \ar@3[ur]_{\twocell{gamma}} 
}$$
}

\resizebox{.5\width}{!}{
$$
 \xymatrix{
\twocell{(s *0 1) *1 (1 *0 m) *1 s *1 m}  \ar@3[r]^{\twocell{tau}} \ar@3[d] &  
\twocell{(s *0 1) *1 (1 *0 m) *1 m} \ar@3[r]^{\twocell{gamma}}&
\twocell{(1*0 m) *1 m}\\
\twocell{(1 *0 s) *1 (m *0 1) *1 m} \ar@3[r]_{\twocell{alpha}}&
\twocell{(1 *0 s) *1 (1 *0 m) *1 m} \ar@3[ur]_{\twocell{tau}}& 
}
$$

$$
\xymatrix{
\twocell{(m *0 2) *1 (s *0 1) *1 (1 *0 m) *1 m } \ar@3[r] \ar@3[d]^{\twocell{gamma}} &
\twocell{(1 *0 s *0 1) *1 (s *0 2) *1 (1 *0 m *0 1) *1 (1 *0 m) *1 m}  \ar@3[r]^{\twocell{alpha}} &
\twocell{(1 *0 s *0 1) *1 (s *0 m) *1 (1 *0 m) *1 m} \ar@3[d] _{\twocell{gamma}} \\
\twocell{(m *0 m ) *1 m} \ar@3[r] _{\twocell{alpha}} &
\twocell{(2 *0 m) *1 (1 *0 m) *1 m}  &
\twocell{(1 *0 s *0 1) *1 (2 *0 m) *1 (1 *0 m) *1 m} \ar@3[l]^{\twocell{gamma}}
}
$$
$$
\xymatrix{
\twocell{(s *0 2) *1 (1 *0 m *0 1) *1 (s *0 1) *1 (1 *0 m) *1 m}\ar@3[d]\ar@3[r]^{\twocell{gamma}}&
\twocell{(s *0 2) *1 (1 *0 m *0 1) *1 (1 *0 m) *1 m} \ar@3[r]^{\twocell{alpha}} & 
\twocell{(s *0 m) *1 (1 *0 m) *1 m}  \ar@3[d]_{\twocell{gamma}} \\
\twocell{(1 *0 s *0 1) *1 (m *0 m) *1 m}   \ar@3[r] _{\twocell{alpha}} & 
\twocell{ (s *0 m ) *1 (1 *0 m) *1 m}\ar@3[r]_{\twocell{gamma}}&
\twocell{(2 *0 m) *1 (1 *0 m) *1 m} 
}
$$

$$
\xymatrix{
\twocell{(s *0 1) *1 (1*0 s) *1 (s *0 1) *1 (1 *0 m) *1 m}\ar@3[d]\ar@3[r]^{\twocell{gamma}}  &
\twocell{(s *0 1) *1 (1*0 s) *1 (1 *0 m) *1 m} \ar@3[r]^{\twocell{tau}}&
\twocell{(s *0 1) *1  (1 *0 m) *1 m} \ar@3[dr]^{\twocell{gamma}}\\
\twocell{(1 *0 s) *1 (s *0 1) *1 (1*0 s) *1 (1 *0 m) *1 m} \ar@3[r]_{\twocell{tau}}&
\twocell{(1 *0 s) *1 (s *0 1) *1 (1 *0 m) *1 m} \ar@3[r]_{\twocell{gamma}}&
\twocell{(1 *0 s) *1 (1 *0 m) *1 m} \ar@3[r]_{\twocell{tau}}&
\twocell{ (1 *0 m) *1 m}
}$$
}

\resizebox{.5\width}{!}{
$$
\xymatrix{
\twocell{(s *0 2) *1 (1*0 s*0 1) *1 (s *0 m) *1 (1 *0 m) *1 m}\ar@3[d]\ar@3[r]^{\twocell{gamma}} &
\twocell{(s *0 2) *1 (1*0 s*0 1) *1 (2 *0 m) *1 (1 *0 m) *1 m} \ar@3[r]^{\twocell{gamma}} & 
\twocell{(s *0 2) *1 (2 *0 m) *1 (1 *0 m) *1 m}   \ar@3[dr]^{\twocell{gamma}}\\ 
\twocell{(1 *0 s *0 1) *1 (s *0 2) *1 (1*0 s*0 1) *1 ( 2 *0 m )*1 (1 *0 m) *1 m} \ar@3[r]_{\twocell{gamma}} &
\twocell{(1 *0 s *0 1) *1 (s *0 2) *1 ( 2 *0 m )*1 (1 *0 m) *1 m} \ar@3[r]_{\twocell{gamma}}&
\twocell{(1 *0 s *0 1) *1 ( 2 *0 m )*1 (1 *0 m) *1 m}\ar@3[r]_{\twocell{gamma}}& 
\twocell{(2 *0 m) *1 (1 *0 m) *1 m}
}$$

$$
\xymatrix{
\twocell{(s *0 1) *1 (1 *0 s) *1 (m *0 1) *1 m}\ar@3[d]_{\twocell{alpha}}\ar@3[r] & 
\twocell{(1 *0 m) *1 (s) *1 m}\ar@3[dr]_{\twocell{tau}}\\
\twocell{(s *0 1) *1 (1 *0 s) *1 (1 *0 m) *1 m}\ar@3[r]_{\twocell{tau}}&
\twocell{(s *0 1) *1 (1 *0 m) *1 m} \ar@3[r]_{\twocell{gamma}}&
\twocell{(1 *0 m) *1 m}
}$$
$$
\xymatrix{
\twocell{(s *0 2) *1 (1 *0 s *0 1) *1 (m *0 m) *1 m} \ar@3[r]^{\twocell{alpha}}\ar@3[d]& 
\twocell{(s *0 2) *1 (1 *0 s *0 1) *1 (2 *0 m) *1 (1 *0 m) *1 m} \ar@3[r]^{\twocell{gamma}}&  
\twocell{(s *0 m) *1 (1 *0 m) *1 m} \ar@3[d]^{\twocell{gamma}}\\
\twocell{(1*0 m *0 1) *1 (s *0 1) *1 (1 *0 m) *1 m} \ar@3[r]_{\twocell{gamma}} & 
\twocell{(1 *0 m *0 1) *1 (1 *0 m) *1 m} \ar@3[r]_{\twocell{alpha}} & 
\twocell{(2 *0 m) *1 (1 *0 m) *1 m}
}$$
}

\item The $18$ \emph{simply-foldable} peaks: 

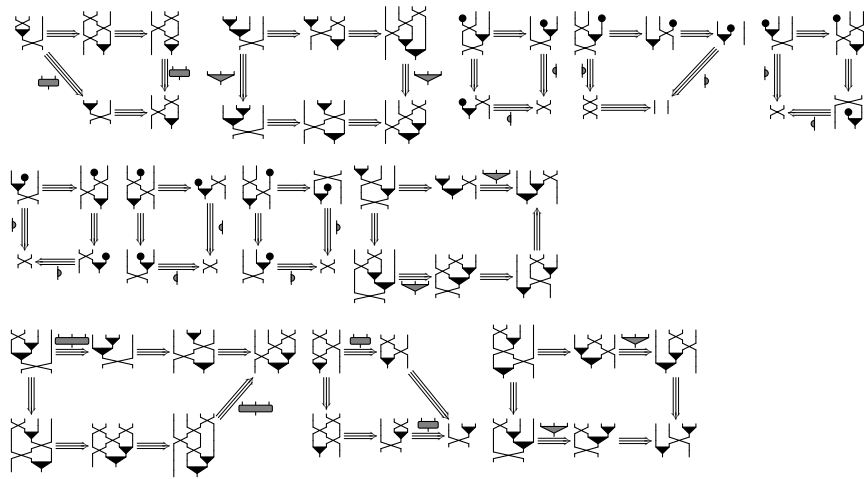
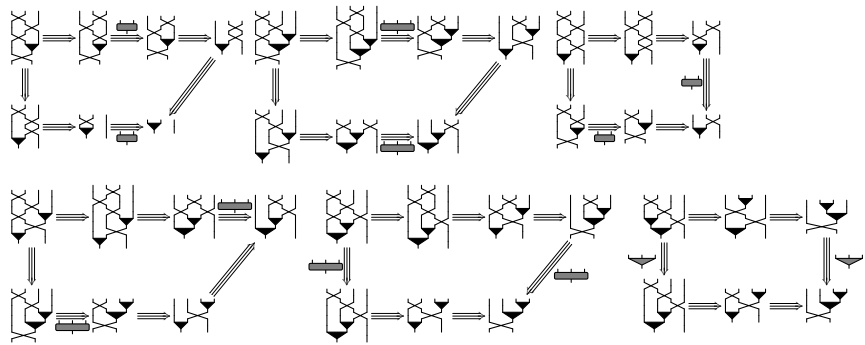
\resizebox{.5\width}{!}{
$$ \xymatrix{
\twocell{(s*0 1) *1 (m *0 1) *1 s}  \ar@3[r] \ar@3[dr]_{\twocell{tau}} & \twocell{( s *0 1) *1 (1 *0 s) *1 (s *0 1) *1 (1 *0 m)}\ar@3[r] & \twocell{( 1 *0 s) *1 (s *0 1) *1 (1 *0 s) *1 (1 *0 m)}\ar@3[d]^{\twocell{tau}}\\
& \twocell{(m*0 1) *1 s} \ar@3[r] & \twocell{(1 *0 s) *1( s *0 1) *1 (1*0 m)}
}$$
$$ 
\xymatrix{
\twocell{(m*0 2) *1 (m *0 1) *1 s} \ar@3[r]\ar@3[d]_{\twocell{alpha}~}& \twocell{(m *0 s) *1 ( s *0 1) *1 (1 *0 m )} \ar@3[r]&  \twocell{(2 *0 s ) *1 (1 *0 s *0 1 ) *1 (s *0 2) *1(1 *0  m *0 1) *1 (1 *0 m)} \ar@3[d]^{~\twocell{alpha}}\\
\twocell{(1*0 m *0 1) *1 (m *0 1) *1 s} \ar@3[r] & \twocell{(1 *0 m *0 1) *1 (1 *0 s ) *1 (s *0 1) *1 (1 *0 m)} \ar@3[r] & \twocell{(2 *0 s ) *1 (1 *0 s *0 1 ) *1 (s *0 m ) *1 (1 *0 m)}
}$$
$$
\xymatrix{
\twocell{(e *0 2) *1 (s *0 1)*1 (1 *0 m)*1 s}	\ar@3[r] \ar@3[d] & \twocell{(1 *0 e *0 1) *1 (1 *0 m ) *1 s }\ar@3[d] ^{~\twocell{l}}\\
\twocell{ (e *0 s) *1 (m *0 1)}    \ar@3[r]_{\twocell{l}}& \twocell{s} 
}$$
$$
\xymatrix{
\twocell{(2 *0 e) *1 (s *0 1)*1 (1 *0 m)*1 s}	\ar@3[r] \ar@3[d]_{\twocell{r}} & \twocell{(2 *0 e ) *1 (1 *0 s ) *1 (m *0 1 )}\ar@3[r] & \twocell{(1 *0 e *0 1)*1 (m *0 1)} \ar@3[ld] ^{~\twocell{r}}\\
\twocell{ s *1 s}    \ar@3[r]& \twocell{2} 
}$$

$$
\xymatrix{
\twocell{(e *0 2) *1 (m *0 1) *1 s} \ar@3[r] \ar@3[d]_{\twocell{r}~} & \twocell{(e *0 2) *1 (1 *0 s ) *1 (s *0 1) *1 (1*0 m)}\ar@3[d]\\
   \twocell{s}& \twocell{s *1 (1*0 e *0 1) *1 (1*0 m)} \ar@3[l]^{\twocell{l}}
}$$
}

\resizebox{.5\width}{!}{

$$\xymatrix{
\twocell{(1*0 e *0 1) *1 (m *0 1) *1 s} \ar@3[r] \ar@3[d]_{\twocell{r}~} & \twocell{(1 *0 e *0 1) *1 (1 *0 s ) *1 (s *0 1) *1 (1*0 m)}\ar@3[d]\\
   \twocell{s}& \twocell{( s *0 e ) *1 (1*0 m)} \ar@3[l]^{\twocell{r}}
}$$

$$
\xymatrix{
\twocell{(1 *0 e *0 1) *1 (s *0 1)*1 (1 *0 s)*1 (m *0 1)} \ar@3[r]  \ar@3[d] & \twocell{(e *0 s)*1 (m *0 1)} \ar@3[d] ^{~\twocell{l}} \\
\twocell{ (1 *0 e *0 1) *1 (1 *0 m) *1 s} \ar@3[r]_{\twocell{l}} & \twocell{s}
}$$

$$
\xymatrix{
\twocell{(2 *0 e) *1 (s *0 1)*1 (1 *0 s)*1 (m *0 1)}  \ar@3[r]  \ar@3[d] & \twocell{(s) *1 (1 *0 e *0 1)*1 (m *0 1)} \ar@3[d] ^{~\twocell{r}} \\
\twocell{ (2 *0 e ) *1 (1 *0 m) *1 s} \ar@3[r]_{\twocell{r}} & \twocell{s}
}$$
$$\xymatrix{
\twocell{(m *0 2) *1 (s *0 1)*1 (1 *0 m)*1 s} \ar@3[r]\ar@3[d] & \twocell{ ( m *0 s) *1 ( m *0 1) } \ar@3[r]^{\twocell{alpha}}&  \twocell{ (2 *0 s ) *1 (1 *0 m *0 1 ) *1 ( m *0  1) } \\
\twocell{( 1*0 s *0 1) *1 ( s *0 2) *1 (1 *0 m *0 1) *1 (1 *0 m) *1 s} \ar@3[r]_{\twocell{alpha}} & \twocell{(1 *0 s *0 1) *1 (s *0 m ) *1 (1 *0 m) *1 s} \ar@3[r] & \twocell{(1 *0 s *0 1) *1 (2 *0 m ) *1 (1 *0 s ) *1 (m *0 1) }
 \ar@3[u]
}$$
}

\resizebox{.5\width}{!}{
$$\xymatrix{
\twocell{(s *0 2) *1 (1 *0 m *0 1) *1 (m *0 1) *1 s} \ar@3[r]^{\twocell{gamma}} \ar@3[d]& \twocell{(1 *0 m *0 1) *1 (m *0 1) *1 s} \ar@3[r] & \twocell{(1 *0 m *0 1) *1 (1 *0 s ) *1 (s *0 1) *1 (1*0 m )}\ar@3[r] & \twocell{(2 *0 s) *1 (1 *0 s *0 1) *1 (s*0 m) *1 ( 1 *0 m)}\\
\twocell{(s *0 2) *1 (1 *0 m *0 1) *1 (1 *0 s) *1 (s*0 1) *1 (1*0 m)} \ar@3[r]& \twocell{(s*0 s) *1 (1*0 s *0 1) *1 (s*0 m) *1 (1*0 m)} \ar@3[r] & \twocell{(2*0 s) *1 ( 1 *0 s *0 1) *1 ( s*0 2) *1 ( 1 *0 s *0 1) *1 (2 *0 m) *1 ( 1*0 m)}\ar@3[ur]_{\twocell{gamma}}
}
$$

$$
\xymatrix{
\twocell{(s *0 1) *1 (1 *0 s) *1 (s *0 1) *1 (m *0 1)  } \ar@3[r]^{\twocell{tau}} \ar@3[d] 
& \twocell{( s *0 1) *1 (1 *0 s) *1 (m *0 1)}\ar@3[dr] \\
\twocell{(1 *0 s ) *1 (s *0 1) *1 (1 *0 s) *1 (m *0 1)} \ar@3[r] 
& \twocell{(1 *0 s) *1( 1 *0 m) *1 s} \ar@3[r]^{\twocell{tau}} 
& \twocell{(1 *0 m) *1 s}
}$$

$$
\xymatrix{
\twocell{(s *0 2) *1 (1 *0 m *0 1) *1 (s *0 1) *1 (1 *0 s) *1(m *0 1)}  \ar@3[r]\ar@3[d]&
\twocell{(1 *0 s *0 1) *1 (m *0 s) *1 (m *0 1)}  \ar@3[r]^{\twocell{alpha}} & 
\twocell{(1 *0 s *0 1) *1 (2 *0 s) *1 (1 *0 m *0 1) *1(m *0 1)}  \ar@3[d] 
\\
\twocell{(s *0 2) *1 (1 *0 m *0 1) *1 (1 *0 m) *1 s}   \ar@3[r]^{\twocell{alpha}} & 
\twocell{(s *0 m) *1 (1 *0 m) *1 s}   \ar@3[r] &
\twocell{(2 *0 m) *1 (1 *0 s) *1 ( m *0 1) }
}$$
}

\resizebox{.5\width}{!}{
$$
\xymatrix{
\twocell{(s *0 1) *1 (1*0 s) *1 (s *0 1)*1 (1 *0 m)*1 s} 	  \ar@3[r]\ar@3[d]&
\twocell{(1 *0 s)*1 (s *0 1) *1 (1*0 s) *1  (1 *0 m)*1 s} 	 \ar@3[r]^{\twocell{tau}} & 
\twocell{(1*0 s) *1 (s *0 1) *1 (1 *0 m)*1 s}  \ar@3[r]		&
\twocell{(1*0 s) *1 (1 *0 s) *1 ( m *0 1)}	\ar@3[dl]\\
\twocell{(s*0 1) *1 (1*0 s) *1 (1 *0 s) *1 (m *0 1)}   \ar@3[r] &
\twocell{(s *0 1) *1 (m *0 1)}\ar@3[r]_{\twocell{tau}} &
\twocell{(m *0 1)}
}$$

$$\xymatrix{
\twocell{(s *0 2) *1 (1*0 s *0 1) *1 (s *0 m)*1 (1 *0 m) *1 s} \ar@3[r] \ar@3[d]&  
\twocell{(1 *0 s *0 1) *1 (s *0 2) *1 (1 *0 s *0 1) *1 (2 *0 m)*1 (1 *0 m) *1 s} \ar@3[r] ^{\twocell{gamma}} &
 \twocell{(1 *0 s *0 1) *1 (s *0 m) *1 (1 *0 m) *1 s }\ar@3[r] &
\twocell{(1 *0 s *0 1) *1 (2 *0 m) *1 (1 *0 s) *1  (m *0 1)  }\ar@3[dl] \\
 \twocell{(s *0 2) *1 (1 *0 s *0 1) *1 (2 *0 m) *1 (1 *0 s ) *1 (m *0 1)} \ar@3[r]& 
\twocell{(s *0 s) *1 (1 *0 m *0 1) *1 (m *0 1)} \ar@3[r]_{\twocell{gamma}} &
\twocell{(2 *0 s) *1 (1 *0 m *0 1) *1 (m *0 1 )}
}$$

$$
\xymatrix{
\twocell{(s *0 1) *1 (1*0 s) *1 (s *0 1) *1 (1 *0 s) *1 (m *0 1)}  \ar@3[r]\ar@3[d]&
\twocell{(1*0 s) *1 (s *0 1) *1 (1 *0 s)*1 (1 *0 s) *1 (m *0 1)} \ar@3[r]& 
\twocell{(1*0 s) *1 (s *0 1) *1 (m *0 1)}  \ar@3[d]_{\twocell{tau}} \\
\twocell{(s *0 1) *1 (1 *0 s ) *1 (1 *0 m )*1 s}  \ar@3[r]_{\twocell{tau}} & 
\twocell{(s *0 1) *1 (1 *0 m )*1 s}   \ar@3[r] &
\twocell{(1 *0 s) *1 ( m*0 1 )} 
}$$
}

\resizebox{.5\width}{!}{
$$
\xymatrix{
\twocell{(s *0 2) *1 (1*0 s *0 1) *1 (s *0 m)*1 (1 *0 s) *1 (m *0 1)} \ar@3[r] \ar@3[d]&  
\twocell{(1*0 s *0 1) *1 (s *0 2) *1 (1*0 s *0 1) *1 (2 *0 m)*1 (1 *0 s) *1 (m *0 1)} \ar@3[r] &
 \twocell{(1 *0 s *0 1) *1 (s*0 s) *1 (1 *0 m *0 1) *1 (m*0 1) }\ar@3[r]^{\twocell{gamma}} &
\twocell{(1 *0 s *0 1) *1 (2*0 s) *1 (1 *0 m *0 1) *1 (m*0 1)  } \\
 \twocell{(s*0 2) *1 (1 *0 s *0 1) *1 (2*0 m) *1 (1 *0 m ) *1 s} \ar@3[r]_{\twocell{gamma}}& 
\twocell{(s*0 m) *1 (1*0 m) *1 s} \ar@3[r] & 
\twocell{(2*0 m ) *1 ( 1 *0 s ) *1 (m*0 1)}\ar@3[ur]
}$$

$$\xymatrix{
\twocell{(s *0 2) *1 (1*0 s *0 1) *1 (s *0 s)*1 (1 *0 m *0 1) *1 (m *0 1)} \ar@3[r] \ar@3[d]_{\twocell{gamma}}&  
\twocell{(1*0 s *0 1) *1 (s *0 2) *1 (1*0 s *0 1) *1 (2 *0 s)*1 (1 *0 m*0 1) *1 (m *0 1)} \ar@3[r] &
 \twocell{(1 *0 s *0 1) *1 (s*0 m) *1 (1 *0 s) *1 (m*0 1) }\ar@3[r] &  \twocell{(1 *0 s *0 1) *1 (2*0 m) *1 (1*0 m) *1 s  } \ar@3[dl]^{\twocell{gamma}}\\
 \twocell{(s*0 2) *1 (1 *0 s *0 1) *1 (2*0 s) *1 (1 *0 m *0 1 ) *1 (m*0 1)} \ar@3[r]& \twocell{(s*0 m) *1 (1*0 s) *1 (m*0 1)} \ar@3[r] & \twocell{(2*0 m ) *1 ( 1 *0 m ) *1 s}
}$$
$$
\xymatrix{
\twocell{(s *0 2) *1 (1*0 s *0 1) *1 (m *0 s)*1 (m *0 1)} \ar@3[r]\ar@3[d]_{\twocell{alpha}~}& \twocell{(1*0 m *0 1) *1 ( s *0 1) *1 (1 *0 s ) *1 (m *0 1)} \ar@3[r]&  \twocell{(1 *0 m *0 1)  *1 (1 *0 m) *1 s} \ar@3[d]^{~\twocell{alpha}}\\
\twocell{(s *0 2) *1 (1*0 s *0 1) *1 (2 *0 s)*1 (1 *0 m *0 1 ) *1 (m *0 1)} \ar@3[r] & \twocell{(s *0 m) *1 (1 *0 s) *1 (m *0 1)} \ar@3[r] & \twocell{(2 *0 m) *1 ( 1 *0 m) *1 s}
}$$
}

\item The $28$ \emph{strongly-foldable} peaks:

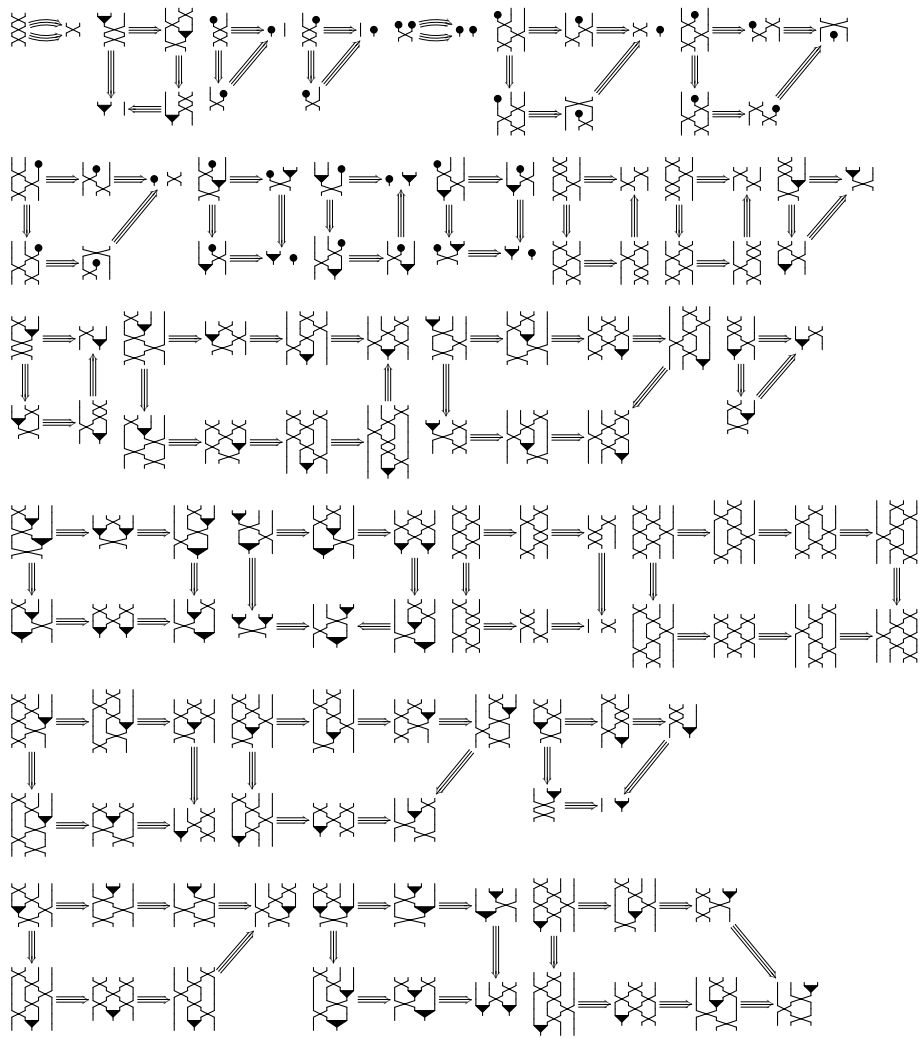
\resizebox{.5\width}{!}{
$$
\xymatrix{
\twocell{s *1 s *1 s }\ar@3@/_/[r] \ar@3@/^/[r] & \twocell{s}
}$$
$$
\xymatrix{
\twocell{(m*0 1) *1 s *1 s} \ar@3[r] \ar@3[d] & \twocell{(1 *0 s ) *1 (s *0 1 ) *1 (1 *0 m )*1 s } \ar@3[d]\\
\twocell{ m*0 1}    & \twocell{(1*0 s)*1(1*0 s) *1 (m*0 1)} \ar@3[l] 
}$$
$$
\xymatrix{
\twocell{(e*0 1) *1 s *1 s } \ar@3[d]\ar@3[r] &  \twocell{e*0 1}\\
\twocell{(1*0 e) *1 s} \ar@3[ur]
}$$
$$
\xymatrix{
\twocell{(1*0 e) *1 s *1 s }\ar@3[r] \ar@3[d]&  \twocell{1*0 e}\\
\twocell{(e*0 1) *1 s} \ar@3[ur]
}$$
$$
\xymatrix{
\twocell{(e*0 e) *1 s }\ar@3@/_/[r] \ar@3@/^/[r] & \twocell{e*0 e}
}$$
$$
\xymatrix{
\twocell{( e *0 2) *1 (s *0 1)*1 (1 *0 s)*1 (s *0 1)}  \ar@3[r] \ar@3[d] & \twocell{(1*0 e *0 1) *1 (1 *0 s ) *1 (s *0 1 )}\ar@3[r] & \twocell{s*0 e}\\
\twocell{( e *0 2) *1 ( 1 *0  s) *1 (s *0 1) *1 (1 *0  s) }    \ar@3[r]& \twocell{s *1 (1 *0 e *0 1 ) *1 (1 *0  s) } \ar@3[ur] 
}$$

$$
\xymatrix{
\twocell{(1 *0 e *0 1) *1 (s *0 1)*1 (1 *0 s)*1 (s *0 1)}  \ar@3[r] \ar@3[d] & \twocell{(e *0 s ) *1 (s *0 1 )}\ar@3[r] & \twocell{s*1 (1 *0 e *0 1)}\\
\twocell{( 1 *0 e *01) *1 ( 1 *0  s) *1 (s *0 1) *1 (1 *0  s) }    \ar@3[r]& \twocell{(s *0 e) *1 (1 *0  s) } \ar@3[ur] 
}$$
}

\resizebox{.5\width}{!}{
$$
\xymatrix{
\twocell{( 2 *0 e) *1 (s *0 1)*1 (1 *0 s)*1 (s *0 1)}  \ar@3[r] \ar@3[d] & \twocell{(1*0 e *0 1) *1 (s *0 1 ) *1 (1 *0 s )}\ar@3[r] & \twocell{e *0 s}\\
\twocell{( 2 *0 e) *1 ( 1 *0  s) *1 (s *0 1) *1 (1 *0  s) }    \ar@3[r]& \twocell{s *1 (1 *0 e *0 1 ) *1 (s *0 1) } \ar@3[ur] 
}$$
$$
\xymatrix{
\twocell{(1 *0 e *0 1) *1 (s *0 1)*1 (1 *0 m)*1 s }	\ar@3[r] \ar@3[d] & \twocell{ (e *0 m ) *1 s }\ar@3[d]\\
\twocell{ (1*0 e *0 1 ) *1 (1 *0 s ) *1 ( m *0 1 )}    \ar@3[r]& \twocell{ m *0 e} 
}$$
$$
\xymatrix{
\twocell{(2 *0 e) *1 (m *0 1) *1 s} \ar@3[r] \ar@3[d] & \twocell{ (e *0 m ) }\\
\twocell{ (2 *0 e ) *1 (1 *0 s ) *1 (s *0 1 )*1 (1 *0 m )}    \ar@3[r]& \twocell{ (1 *0 e *0 1) *1 (s *0 1 )*1 (1 *0 m )}\ar@3[u] 
}$$
$$
\xymatrix{
\twocell{(e *0 2) *1 (s *0 1)*1 (1 *0 s)*1 (m *0 1)} 	\ar@3[r] \ar@3[d] & \twocell{ (1 *0 e *0 1 ) *1 ( 1 *0 s) *1 (m *0 1) }\ar@3[d]\\
\twocell{ (e *0 m ) *1 s}    \ar@3[r]& \twocell{ m *0 e} 
}$$
$$
\xymatrix{
\twocell{(s *0 1) *1 (s *0 1)*1 (1 *0 s)*1 (s *0 1)} 	 \ar@3[r] \ar@3[d] & \twocell{ (1 *0 s )*1 (s *0 1) }\\
\twocell{ (s *0 1 )*1 (1 *0 s )*1 (s *0 1) *1 (1 *0 s)}    \ar@3[r]& \twocell{  (1 *0 s )*1 (s *0 1) *1 (1 *0 s ) *1 (1 *0 s )}\ar@3[u] 
}
$$
$$
\xymatrix{
\twocell{(s *0 1)*1 (1 *0 s)*1 (s *0 1) *1 (s *0 1)}  	 \ar@3[r] \ar@3[d] & \twocell{  (s *0 1) *1 (1 *0 s ) }\\
\twocell{ (1 *0 s) *1 (s *0 1 )*1 (1 *0 s )*1 (s *0 1)}    \ar@3[r]& \twocell{ (1 *0 s ) *1  (1 *0 s )*1 (s *0 1) *1 (1 *0 s )}\ar@3[u] 
}$$
$$
\xymatrix{
\twocell{(s *0 1) *1 (s *0 1)*1 (1 *0 m)*1 s}  \ar@3[d]\ar@3[r] &  \twocell{ ( m *0 1 ) *1 s }\\
\twocell{ ( s *0 1 ) *1 ( 1 *0  s ) *1 ( m *0 1 )} \ar@3[ur]
}$$
}

\resizebox{.5\width}{!}{
$$
\xymatrix{
\twocell{(s *0 1 ) *1 (1 *0 m)*1 s *1 s}   	 \ar@3[r] \ar@3[d] & \twocell{ (s *0 1) *1 (1 *0 m )} \\
\twocell{ (1 *0 s ) *1 ( m *0 1 )*1 s}    \ar@3[r]& \twocell{  (1 *0 s ) *1  (1 *0 s )*1 (s *0 1 ) *1 (1 *0 m )}\ar@3[u] 
}$$
$$
\xymatrix{
\twocell{(s *0 2) *1 (1 *0 m *0 1) *1 (s *0 1)*1 (1 *0 s)*1 (s *0 1)} \ar@3[r] \ar@3[d]&  
\twocell{(1 *0 s *0 1) *1 (m *0 s) *1 (s *0 1 )}\ar@3[r] &
 \twocell{(1 *0 s *0 1) *1 ( 2 *0 s) *1 (1 *0 s *0 1) *1 (s *0 2 ) *1 (1 *0 m *0 1) }\ar@3[r] & 
 \twocell{(2 *0 s) *1 (1 *0 s *0 1) *1 ( s *0 s ) *1 (1 *0 m *0 1) } \\
 \twocell{(s*0 2) *1 (1 *0 m *0 1) *1 (1*0 s) *1 (s *0 1 ) *1 (1 *0 s)} \ar@3[r]& 
 \twocell{(s*0 s) *1 (1*0 s *0 1) *1 (s *0 m) *1 (1 *0 s)} \ar@3[r] & 
 \twocell{(s *0 s) *1 (1 *0 s *0 1) *1 ( s *0 2 ) *1 (1 *0 s *0 1) *1  (2 *0 s) *1 (1 *0 m *0 1) } \ar@3[r]&
 \twocell{ (2 *0 s) *1 (1 *0 s *0 1)  *1 ( s *0 2 ) *1(1 *0 s *0 1) *1(1 *0 s *0 1)*1 (2 *0 s)*1 (1 *0 m *0 1) }  \ar@3[u]
 }$$
$$
\xymatrix{
\twocell{(m *0 2) *1 (s *0 1) *1 (1 *0 s) *1 (s *0 1) } \ar@3[r]\ar@3[d]&
\twocell{(1 *0 s *0 1) *1 (s *0 2) *1 (1 *0 m *0 1) *1 (1 *0 s) *1 (s *0 1)} \ar@3[r]& 
\twocell{(1 *0 s *0 1) *1 (s *0 s) *1 (1 *0 s *0 1) *1 (s *0 m)}  \ar@3[r] &
\twocell{(1 *0 s *0 1) *1 (2 *0 s) *1 (1 *0 s *0 1) *1 (s *0 2)*1 (1 *0 s *0 1) *1 (2 *0 m)}  \ar@3[dl] \\
\twocell{(m *0 s) *1 (s *0 1) *1 (1 *0 s) }   \ar@3[r] & 
\twocell{(2 *0 s)  *1 (1 *0 s *0 1) *1 (s *0 2) *1 (1 *0 m *0 1) *1 (1 *0 s)} \ar@3[r] &
\twocell{(2 *0 s)  *1 (1 *0 s *0 1) *1 (s *0  s) *1 (1 *0 s *0 1) *1 (2 *0 m)}
}
$$
$$
\xymatrix{
\twocell{(s *0 1) *1 (s *0 1) *1 (1 *0 s) *1 (m *0 1)} \ar@3[r]  \ar@3[d]&  \twocell{(1*0 s) *1 (m *0 1)}\\
\twocell{(s *0 1) *1 (1 *0 m) *1 s} \ar@3[ur]}
$$
}

\resizebox{.5\width}{!}{
$$
\xymatrix{
\twocell{(s *0 2) *1 (1 *0 m *0 1) *1 (s *0 1)*1 (1 *0 m)*1 s} \ar@3[r]\ar@3[d]& \twocell{(1 *0 s *0 1) *1 (m *0 m) *1 (s)} \ar@3[r]&  \twocell{(1 *0 s *0 1) *1 (2 *0 m ) *1 (1*0 s) *1 (s *0 1) *1 (1 *0 m) } \ar@3[d]\\
\twocell{(s *0 2) *1 (1*0 m *0 1) *1 (1 *0 s)*1 (m *0 1 )} \ar@3[r] & \twocell{(s *0 s) *1 (1 *0 s*0 1) *1 (m *0 m)} \ar@3[r] & \twocell{(2 *0 s) *1 ( 1 *0 m *0 1) *1 (s *0  1) *1 (1 *0 m)}
}
$$
$$
\xymatrix{
\twocell{(m  *0 2) *1 (s *0 1) *1 (1 *0 s) *1 (m *0 1)} \ar@3[r] \ar@3[d] &
\twocell{(1 *0 s *0 1) *1 (s *0 2) *1 (1 *0 m *0 1) *1 (1 *0 s) *1 (m *0 1)} \ar@3[r]&
\twocell{(1 *0 s *0 1) *1 (s *0 s) *1 (1 *0 s *0 1) *1 (m *0 m)}\ar@3[d]\\
\twocell{(m *0 m) *1 s} \ar@3[r]&
\twocell{(2 *0 m) *1 (1 *0 s) *1 (s *0 1) *1 (1 *0 m)}&
\twocell{( 1 *0 s *0 1) *1 (2 *0 s) *1 (1 *0 m *0 1) *1 (s *0 1) *1 (1 *0 m)} \ar@3[l]
}$$
$$
\xymatrix{
\twocell{(s *0 1) *1 (1*0 s) *1 (s *0 1)*1 (1 *0 s)*1 (s *0 1)}  \ar@3[r]\ar@3[d]&
\twocell{(1*0 s)*1 (s *0 1) *1 (1*0 s) *1 (1 *0 s)*1 (s *0 1)} \ar@3[r] &  
\twocell{(1 *0 s) *1 (s*0 1) *1  (s *0 1)}  \ar@3[d]\\
\twocell{(s *0 1) *1 (1*0 s) *1 (1*0 s) *1 (s *0 1)*1 (1 *0 s)}  \ar@3[r] & 
\twocell{(s *0 1) *1 (s*0 1)*1 (1 *0 s)} \ar@3[r] & 
\twocell{(1 *0 s)}
}$$
$$
\xymatrix{
\twocell{(s *0 2) *1 (1*0 s *0 1) *1 (s *0 s)*1 (1 *0 s *0 1) *1 (s *0 2)}	  \ar@3[r]\ar@3[d]&
\twocell{(1*0 s *0 1) *1 (s *0 2) *1 (1*0 s *0 1)*1 (2 *0 s)*1 (1 *0 s *0 1) *1 (s *0 2)}	 \ar@3[r] & 
\twocell{(1*0 s *0 1) *1 (s *0 2) *1 (2 *0 s)*1 (1 *0 s *0 1) *1 (s *0 s)}  \ar@3[r]&
\twocell{(1*0 s *0 1) *1 (2 *0 s) *1 (1*0 s *0 1)*1 (s *0 2)*1 (1 *0 s *0 1) *1 (2 *0 s)}	\ar@3[d]\\
\twocell{(s *0 2) *1 (1*0 s *0 1)*1 (2 *0 s)*1 (1 *0 s *0 1) *1 (s *0 2)*1 (1 *0 s *0 1)}  \ar@3[r] &
 \twocell{(s *0 s) *1 (1 *0 s*0 1)*1 (s *0 s) *1 (1 *0 s *0 1)} \ar@3[r] &
\twocell{(2 *0 s) *1 (1*0 s *0 1) *1 (s *0 2) *1 (1*0 s *0 1)*1 (2 *0 s)*1 (1 *0 s *0 1)} \ar@3[r]&
\twocell{(2 *0 s) *1 (1*0 s *0 1) *1 (s *0 s) *1 (1 *0 s *0 1)*1 (2 *0 s)}
}$$
}

\resizebox{.5\width}{!}{
$$
\xymatrix{
\twocell{(s *0 2) *1 (1*0 s *0 1) *1 (s *0 m)*1 (1 *0 s)*1 (s *0 1)}	  \ar@3[r]\ar@3[d]&
\twocell{(1*0 s *0 1) *1 (s *0 2) *1 (1 *0 s *0 1) *1 (2 *0 m) *1 (1 *0 s) *1 (s *0 1) }  \ar@3[r]&
\twocell{(1*0 s *0 1) *1 (s *0 s) *1 (1 *0 m *0 1)*1 (s *0 1)}	\ar@3[d]\\
\twocell{(s *0 2) *1 (1*0 s *0 1)*1 (2 *0 m)*1 (1 *0 s) *1 (s *0 1)*1 (1 *0 s)}  \ar@3[r] &
\twocell{(s *0 s) *1 (1 *0 m *0 1)*1 (s *0 1)*1 (1 *0 s)} \ar@3[r] &
\twocell{(2 *0 s) *1 (1*0 s *0 1) *1 (m *0 s)}
}$$

$$
\xymatrix{
\twocell{(s *0 2) *1 (1*0 s *0 1) *1 (s *0 s)*1 (1 *0 m *0 1) *1 (s *0 1)}	  \ar@3[r]\ar@3[d]&
\twocell{(1*0 s *0 1) *1 (s *0 2) *1 (1*0 s *0 1) *1 (2 *0 s)*1 (1 *0 m *0 1) *1 (s *0 1)}	 \ar@3[r] & 
\twocell{(1*0 s *0 1) *1 (s *0 m) *1 (1 *0 s) *1 (s *0 1)}  \ar@3[r]&
\twocell{(1*0 s *0 1) *1 (2 *0 m) *1 (1 *0 s) *1 (s *0 1)*1 (1 *0 s)} \ar@3[dl]\\
\twocell{(s *0 2) *1 (1*0 s *0 1) *1 (2 *0 s)*1 (1 *0 s *0 1) *1 (m *0 2)}	  \ar@3[r] &
\twocell{(s *0 s)*1 (1 *0 s *0 1) *1 (m *0 s)} \ar@3[r] &
\twocell{(2 *0 s) *1 (1*0 m *0 1) *1 (s *0 1) *1 (1 *0 s)}
}$$
$$
\xymatrix{
\twocell{(s *0 1) *1 (1*0 s) *1 (m *0 1) *1 s}  \ar@3[r] \ar@3[d] &
\twocell{(s *0 1)*1 (1*0 s) *1 (1*0 s) *1 (s *0 1) *1 (1 *0 m)} \ar@3[r] &
\twocell{(s *0 1)*1 (s *0 1) *1 (1 *0 m)} \ar@3[dl]\\
\twocell{( 1 *0 m) *1  s *1 s  }    \ar@3[r]&
\twocell{1*0 m} 
}$$
}

\resizebox{.5\width}{!}{
$$
\xymatrix{
\twocell{(s *0 2) *1 (1*0 s *0 1) *1 (m *0 s)*1 (s *0 1)}  \ar@3[r]\ar@3[d]&
 \twocell{(1 *0 m *0 1)*1 (s *0 1) *1 (1 *0 s) *1 (s *0 1)} \ar@3[r]& 
 \twocell{(1 *0 m *0 1)*1 (1 *0 s) *1 (s *0 1) *1 (1 *0 s)}  \ar@3[r]&
\twocell{(2 *0 s) *1 (1 *0 s *0 1) *1 ( s *0 m) *1 (1 *0 s)}   \\
\twocell{(s *0 2) *1 (1 *0 s *0 1) *1 (2 *0 s) *1 (1 *0 s *0 1) *1 (s *0 2) *1 (1 *0 m *0 1)}  \ar@3[r] & 
\twocell{(s *0 s) *1 (1 *0 s *0 1) *1 (s *0 s) *1 (1 *0 m *0 1)} \ar@3[r] &
\twocell{(2 *0 s) *1 (1 *0 s *0 1) *1 (s *0 2) *1 (1 *0 s *0 1 ) *1 (2 *0 s) *1 (1 *0 m *0 1)}\ar@3[ur]
}$$

$$
\xymatrix{
\twocell{(s *0 2) *1 (1 *0 s *0 1) *1 (m *0 m)*1 s}  \ar@3[r]\ar@3[d]&
\twocell{(1 *0 m *0 1) *1 (s *0 1) *1 (1 *0 m) *1 s} \ar@3[r]& 
\twocell{(1 *0 m *0 1) *1 (1 *0 s) *1 (m *0 1)}  \ar@3[d] \\
\twocell{(s *0 2) *1 (1 *0 s *0 1) *1 (2 *0 m )*1 (1 *0 s) *1 (s*0 1 ) *1 (1 *0 m)}  \ar@3[r] & 
\twocell{(s *0 s) *1 (1 *0 m *0 1) *1 (s*0 1 ) *1 (1 *0 m)}  \ar@3[r] &
\twocell{(2 *0 s) *1 (1 *0 s *0 1) *1 ( m*0 m )} 
}
$$
$$
\xymatrix{
\twocell{(s *0 2) *1 (1*0 s *0 1) *1 (s *0 s)*1 (1 *0 s *0 1) *1 (m *0 2)}	 \ar@3[r]\ar@3[d]&
\twocell{(s *0 2) *1 (1 *0 s *0 1) *1 (2 *0 s) *1 (1 *0 m *0 1) *1 (s *0 1)} \ar@3[r]& 
\twocell{(s *0 m) *1 (1 *0 s) *1 (s *0 1)}  \ar@3[dr] \\
\twocell{(1 *0 s *0 1) *1 (s *0 2 )*1 (1 *0 s *0 1) *1 (2 *0 s) *1 (1 *0 s *0 1)*1 (m *0 2)}  \ar@3[r] & 
\twocell{(1 *0 s *0 1) *1 (s *0 s )*1 (1 *0 s *0 1) *1 (m *0 s)}  \ar@3[r] &
\twocell{(1 *0 s *0 1) *1 (2 *0 s) *1 (1 *0 m *0 1) *1 (s *0 1) *1 (1 *0 s) } \ar@3[r] &
\twocell{(2*0 m) *1 ( 1 *0 s) *1 (s *0 1) *1 (1 *0 s)}
}$$
}

\end{itemize}

\newpage

\section{Kelly's peaks $4$-cells}\label{CellyKelly}

In this appendix we  give a list of possible ways to define a $4$-cell  with border the solution of a Kelly-peak or weak-Kelly-peak. These $4$-cells are implicitly defined by given a decomposition of another $4$-cell.

In order to reduce the size of pictures, we note \doublebox{$\D$} the $4$-cell which border the given solution of a critical peak represented by its source $\D$.

\begin{itemize}

%%%%%%%%%%%%%%%%%%%%%%%%%%%%%
\item \twocell{(e *0 2) *1 (m *0 1)*1 m}~:

\resizebox{.5\width}{!}{
$$ \scalebox{0.9}{\xymatrix{
&& {\twocell{(e *0 e*0 2)*1 (( 1 *0 m *0 1) *1 m *0 1) *1 m}} \ar@3^{~~~~~~~~~~~~~~~~=}[d] \ar@3[rr]&& {\twocell{(e *0 e*0 2)*1 ((1 *0 m *0 1) *1 1 *0 m) *1 m}} \ar@3@/^2pc/[dr] \ar@3[d]^{~~~~~~\doublebox{T}} \\
\twocell{(e *0 e*0 2)*1 penta}= & \twocell{(e *0 e*0 2)*1 (( m *0 2) *1 m *0 1) *1 m} \ar@3@/^2pc/[ur]_{\twocell{~~(e *0 2) *1 (tria *0 1) *1 m}} \ar@3[drr]^{~~~~~~~~~~~~~~=}\ar@3[r]& \twocell{(e *0 2) *1 (m *0 1) *1 m}\ar@3[rr]&& \twocell{(e*0 m) *1 m} \ar@3[r]& {\twocell{(e *0 e*0 2)*1 (( 2 *0 m ) *1 (1 *0 m)) *1 m}} & \mbox{{\huge where} }\; \doublebox{T}=\\
&&&{\twocell{(e *0 e*0 2)*1 ( m *0 m)  *1 m }}\ar@3@/_2pc/[urr]\ar@3[ur]_{~~\twocell{(e *0 m) *1 tria}}
}}$$

$$ \scalebox{0.9}{\xymatrix{
&& {\twocell{( e *0 2 ) *1 (m *0 1)*1 (e *0 m) *1 m}}  \ar@3@/_2pc/[ddll] \ar@3@/^2pc/[ddrr]\ar@3[d]\\
&=& \twocell{(e *0 2) *1 (m *0 1) *1 m} \ar@3[dl]\ar@3[dr]&=\\
\twocell{ (e *0 m) *1 m} \ar@3[r] & \twocell{m} && \twocell{(e *0 m ) *1 m} \ar@3[ll]&  \twocell{ ( e *0 m ) *1 (e *0 m) *1 m} \ar@3[l]\ar@3@/^2pc/[llll]_{\raisebox{1ex}{=}}
}} \; ;$$
}

%%%%%%%%%%%%%%%%%%%%%%%%%%%%%

\item \twocell{(2 *0 e) *1 (m *0 1)*1 m}~:

\resizebox{.5\width}{!}{
$$ \scalebox{0.9}{ \xymatrix{
&& {\twocell{(2 *0 e *0 e)*1 (( 1 *0 m *0 1) *1 m *0 1) *1 m}} \ar@3[d]^{~~~~~~~~~~~~~~~~=}_{\doublebox{T'}~~~~~} \ar@3[rr]&& {\twocell{(2 *0 e *0 e)*1 ((1 *0 m *0 1) *1 1 *0 m) *1 m}} \ar@3@/^2pc/[dr] \ar@3[d]^{~~\twocell{(2*0 e)*1 (1 *0 tria) *1 m}} \\
\twocell{(2 *0 e *0 e)*1 penta} = &\twocell{(2 *0 e *0 e)*1 (( m *0 2) *1 m *0 1) *1 m} \ar@3@/^2pc/[ur]\ar@3[drr]^{~~~~~~~~~~~~~~=} \ar@3[r] & \twocell{(m *0 e) *1 m}\ar@3[rr]&& \twocell{(2*0 e)*1 (1*0 m) *1 m} \ar@3[r]& {\twocell{(2 *0 e *0 e)*1 (( 2 *0 m ) *1 (1 *0 m)) *1 m}}& 
\mbox{\huge{where} } \; \doublebox{T'} = \\
&&&{\twocell{(2 *0 e *0 e)*1 ( m *0 m)  *1 m }}\ar@3@/_2pc/[urr]\ar@3[ur]_{~~\twocell{(m *0 e) *1 tria}}
}}$$

$$ \scalebox{0.9}{\xymatrix{
&& {\twocell{(2 *0 e *0 e)*1 (( 1 *0 m *0 1) *1 m *0 1) *1 m}}  \ar@3@/_2pc/[ddll] \ar@3@/^2pc/[ddrr]\ar@3[d]\\
&=& \twocell{(2 *0 e) *1 (m *0 1)*1 m} 	\ar@3[dl]\ar@3[dr]&=\\
\twocell{(2 *0 e *0 e)*1 (( m *0 2) *1 m *0 1) *1 m} \ar@3@/_2pc/[rrrr]^{\raisebox{1 ex} {=}}\ar@3[r] &  \twocell{(m *0 e) *1 m} \ar@3[rr]&&\twocell{m}& \twocell{(m *0 e) *1 m} \ar@3[l]
}} \; ;$$
}

%%%%%%%%%%%%%%%%%%%%%%%%%%%%%

\begin{multicols}{2}

\item  \twocell{(e *0 e) *1 m}~:

\resizebox{.5\width}{!}{
$$ \scalebox{0.9}{\xymatrix{
&{\twocell{((e *0 e *0  e)*1 (m *0 1)) *1 m}} \ar@3@/^1pc/[dddr]  \ar@3@/_1pc/^{ ~~= }[dddr] \ar@3[rrrr] &&&&
{\twocell{((e *0 e *0  e)*1 (1 *0 m))*1 m}}\ar@3[dl] \ar@3@/_3pc/_{\twocell{e *0 e}~ \circ ~\doublebox{\twocell{(m *0 e) *1 m}}} ^{ ~~= }[dddlll]  \ar@3@/^4pc/_{ ~~= }[dddlll]\\
\twocell{(e *0 e) *1 tria} = &&&& \twocell{(e *0 e) *1 m} \ar@3[dl]\\
&&& \twocell{e} \\
&&{~~ \twocell{(e *0 e) *1 m}~~} \ar@3@/^1pc/[ur]  \ar@3@/_1pc/[ur]
}} \; ;$$
}

\item  \twocell{(1*0 e) *1 s *1 m}~: if \resizebox{.8\width}{!}{$\doublebox{ T''}=\twocell{(e *0 2)} ~ \circ \doublebox{\twocell{(m *0 e) *1 m}}$} \;, then

\resizebox{.5\width}{!}{
$$ \scalebox{0.9}{ \xymatrix{
&{\twocell{((e *0 1 *0 e) *1 (s *0 1) *1 (m *0 1)) *1 m}} \ar@3[ddd]^{~~~~~~~ \comm  } \ar@3[dr]^{~~~~~~~~~~~~~~=}   \ar@3[rrrr] &&&& 
{\twocell{(e *0 1)*1 s *1 m}} \ar@3[dll] \ar@3[ddl]
\\
\twocell{(e *0 1) *1 s} ~ \circ \doublebox{\twocell{(m *0 e) *1 m}} = &&\twocell{(1 *0 e) *1 (m *0 e) *1 m} \ar@3[d]\ar@3[r]& 
\twocell{(e *0 1) *1 m} \ar@3[d]& 
\\
&&\twocell{(e *0 1) *1 (e *0 m) *1 m}  \ar@3@/_1pc/^{ ~~\doublebox{T'' }} [ur]  \ar@3@/_3pc/^{\raisebox{2 ex}{=} }[rr]& 
\twocell{1} & \twocell{(1 *0 e) *1 m} \ar@3[l] 
\\
&\twocell{(e *0 1) *1 (s *0 e) *1 (1 *0 m) *1 m} \ar@3@/_7pc/[uuurrrr]\ar@3[ur]_{~~~~~~=} 
}}$$
}

\columnbreak

\item \twocell{(e*0 1) *1 s *1 m}	~: 

\resizebox{.5\width}{!}{
$$ \scalebox{0.9}{\xymatrix{
&
\twocell{(e *0 1 *0 e) *1 (1 *0 s) *1 (m *0 1) *1 m} \ar@3[ddd]^{~~~~~~~ \comm } \ar@3[dr]^{~~~~~~~~~~~~~~=}   \ar@3[rrrr] &&&&
 {\twocell{(1 *0 e)*1 s *1 m}} \ar@3[dll] \ar@3[ddl]
 \\
 \twocell{(1 *0 e) *1 s} ~ \circ \doublebox{\twocell{(e *0 2) *1 (m *0 1) *1 m}} = &
&\twocell{(e *0 e *0 1) *1 (m *0 1) *1 m} \ar@3[d]\ar@3[r]& 
\twocell{(e *0 1) *1 m} \ar@3[d]& 
\\
&&\twocell{(e *0 1) *1 (e *0 m) *1 m}  \ar@3[ur]^{\twocell{(e *0 1) *1 tria} }  \ar@3@/_3pc/^{\raisebox{2 ex}{=} }[rr]& 
\twocell{1} & \twocell{(1 *0 e) *1 m} \ar@3[l] \\
&
\twocell{(e *0 1 *0 e) *1 (1 *0 s) *1 (1 *0 m) *1 m} \ar@3@/_7pc/[uuurrrr]\ar@3[ur]_{~~~~~~=} 
}} $$
}

\item \twocell{(2 *0 e) *1 (s *0 1)*1 (1 *0 m)*1 m} ~: 

\resizebox{.5\width}{!}{
$$ \scalebox{0.9}{\xymatrix{
&\twocell{(2 *0 e) *1 ( s *0 1) *1 (m *0 1) *1 m } \ar@3[rrr]  \ar@3^{~~~~\comm}[dd] \ar@3^{~~~~~~~~~\comm}[dr] &&&
\twocell{(2 *0 e) *1 (s *0 1)*1 (1 *0 m)*1 m} \ar@3[dll] \ar@3[dd]\\
\twocell{( 2 *0 e )*1 g} +
&& \twocell{s *1 m} \ar@3[r]& \twocell{m}\\
&\twocell{(2 *0 e) *1 (m *0 1) *1 m}\ar@3@/^/[urr]\ar@3[rrr]^{\raisebox{3 ex}{\doublebox{\twocell{ (m *0 e) *1 m}}}} &&&
\twocell{(2 *0 e) *1 (1 *0 m) *1 m} \ar@3[ul]
}}$$
}

\item \twocell{(1 *0 e *0 1) *1 (s *0 1)*1 (1 *0 m)*1 m}~:

\resizebox{.5\width}{!}{
$$ \scalebox{0.9}{\xymatrix{
&\twocell{(1 *0 e *0 1) *1 ( s *0 1) *1 (m *0 1) *1 m } \ar@3[rrr]  \ar@3[ddd] \ar@3^{~~~~~~~~~~~~~~\comm}[dr]&&&
\twocell{(1 *0 e *0 1) *1 (s *0 1)*1 (1 *0 m)*1 m} \ar@3[dl] \ar@3[ddd]
\\
\twocell{(1 *0 e *0 1 )*1 g}= && \twocell{(e*0 2) *1 (m *0 1) *1 m} \ar@3[r] \ar@3@/_2pc/^{\raisebox{1 ex}{\doublebox{\twocell{(e *0 2) *1 (m *0 1)*1 m}}}}[dr]&
\twocell{(e *0 m) *1 m}\ar@3[d]
\\
&&&\twocell{m}
\\
&\twocell{(1 *0 e *0 1) *1 (m *0 1) *1 m} \ar@3@/_1pc/[urr] \ar@3^{\raisebox{7 ex}{$\doublebox{ \twocell{(1*0 e) *1 s *1 m} } \circ \twocell{m}$~~~~~~~~~~~~~~~~~~~~}}[rrr] &&&
\twocell{(1 *0 e *0 1) *1 (1 *0 m) *1 m} \ar@3[ul]^{\twocell{tria}~~~~~~}
}}$$
}

\item \twocell{(e *0 2) *1 (s *0 1)*1 (1 *0 m)*1 m} ~:

\resizebox{.5\width}{!}{
$$ \scalebox{0.9}{\xymatrix{
&\twocell{(e *0 2) *1 ( s *0 1) *1 (m *0 1) *1 m } \ar@3[rrr]  \ar@3[ddd] \ar@3^{~~~~~~~~~~~~~~\comm}[dr]&&&
\twocell{(e *0 2) *1 (s *0 1)*1 (1 *0 m)*1 m} \ar@3[dl] \ar@3[ddd]
\\
 \twocell{( e *0 2 )*1 g}= && \twocell{(1 *0 e *0 1) *1 (m *0 1) *1 m} \ar@3[r]    \ar@3@/_2pc/^{\raisebox{2 ex}{~~~\twocell{tria}}}[dr]&
\twocell{(1 *0 e *0 1) *1 (1 *0 m) *1 m}\ar@3[d]
\\
&&& \twocell{m}
\\
&\twocell{(e *0 2) *1 (m *0 1) *1 m}\ar@3@/_1pc/[urr] \ar@3^{\raisebox{7 ex}{$\doublebox{ \twocell{(e*0 1) *1 s *1 m} } \circ \twocell{m}$~~~~~~~~~~~~~~~~~~~~}}[rrr]^{~~~~~~~~~~~~~~~~~~~~~~~~~~\raisebox{4 ex}{\doublebox{\twocell{(e *0 2) *1 (m *0 1) *1 m}}}} &&&
\twocell{(e *0 m) *1 m} \ar@3@/_2pc/[ul]
} }$$
}

\item \twocell{(s *0 2) *1 (1 *0 m *0 1) *1 (m *0 1) *1 m}~: 

\resizebox{.5\width}{!}{
$$ \scalebox{0.9}{\xymatrix{
&
\twocell{(s *0 2) *1 (m *0 2) *1 (m *0 1) *1 m}  \ar@3[r] \ar@3@/^4pc/[rrr]_{\raisebox{5 ex}{~}=} \ar@3[dd]^{~~\twocell{(s *0 2 ) *1 penta}} &
\twocell{(s *0 2 ) *1 (m *0 m ) *1 m} \ar@3[d] \ar@3[r]&
\twocell{(m *0 m ) *1 m}  \ar@3[d]&
\twocell{(m *0 2) *1 (m *0 1) *1 m}  \ar@3[dd]_{\twocell{penta}~~} \ar@3[l]
\\
\twocell{(g *0 1) *1 m} =
&&\twocell{ (s *0 m ) *1 (1 *0 m) *1 m} \ar@3[r]^{\raisebox{6 ex}{\twocell{(2 *0 m ) *1 g}}}&
\twocell{(2 *0 m) *1 (1 *0 m) *1 m} & 
\\ 
&
\twocell{(s *0 2) *1 (1 *0 m *0 1) *1 (m *0 1) *1 m}  \ar@3[r] \ar@3@/_3pc/[rrr]&
\twocell{(s *0 2) *1 (1 *0 m *0 1) *1 (1 *0 m) *1 m}    \ar@3[u] &
\twocell{(1 *0 m *0 1) *1 (1 *0 m) *1 m}  \ar@3[u]  & 
\twocell{(1 *0 m *0 1) *1 (m *0 1) *1 m} \ar@3[l]
}} $$
}

\item \twocell{(s *0 1) *1 (s *0 1) *1 (1 *0 m) *1 m} ~: 

\resizebox{.5\width}{!}{
$$ \scalebox{0.9}{\xymatrix{
\twocell{(s *0 1) *1 g}= &\twocell{(s *0 1) *1 (s *0 1) *1 (m *0 1) *1 m} \ar@3[dr]^{~~~~~~~\twocell{g}} \ar@3[rr] &&
\twocell{(s *0 1) *1 (s *0 1) *1 (1 *0 m) *1 m}  \ar@3[r] \ar@3[d] & 
\twocell{(1*0 m) *1 m}
\\
&&
\twocell{(s *0 1) *1 (m *0 1) *1 m} \ar@3[r] &
\twocell{(s *0 1) *1 (1 *0 m) *1 m} \ar@3[ur]
}}$$
}

\item \twocell{(s *0 1) *1 (1 *0 m) *1 (s) *1 m }~:  

\resizebox{.5\width}{!}{
$$ \scalebox{0.9}{
 \xymatrix{
&&\twocell{(s *0 1) *1 (1 *0 m) *1 s *1 m}  \ar@3[rr] \ar@3[d] &  &
\twocell{(s *0 1) *1 (1 *0 m) *1 m} \ar@3[d]
\\
\twocell{s *0 1} \circ \doublebox{\twocell{(s *0 1) *1 (1 *0 s) *1 (m *0 1) *1 m}} =&
\twocell{(s *0 1) *1 (s *0 1) *1 (1 *0 s) *1 (m *0 1) *1 m} \ar@3[ur] \ar@3@/_2pc/[rrd]&
\twocell{(1 *0 s) *1 (m *0 1) *1 m} \ar@3[r]&
\twocell{(1 *0 s) *1 (1 *0 m) *1 m} \ar@3[r]&
\twocell{(1*0 m) *1 m}
\\
&&&\twocell{(s *0 1) *1(s *0 1) *1 (1 *0 s) *1 (1 *0 m) *1 m} \ar@3[rr] \ar@3[u]_{~~~~~~~~\comm}^{\comm~~~~~~~~~~~~}
&& \twocell{(s *0 1) *1(s *0 1) *1 (1 *0 m) *1 m} \ar@3@/_5pc/[uul]^{\doublebox{\twocell{(s *0 1) *1 (s *0 1) *1 (1 *0 m) *1 m}}~~~~~~} \ar@3[ul]
}} $$
}

\item \twocell{(m *0 2) *1 (s *0 1) *1 (1 *0 m) *1 m} ~:

\resizebox{.5\width}{!}{
$$ \scalebox{0.9}{\xymatrix{ &
&
\twocell{(m *0 2) *1 (s *0 1) *1 (m *0 1) *1 m } \ar@3[d]^{~~~~~~~\comm} \ar@3[r] \ar@3[dl]& 
\twocell{(1 *0 s *0 1) *1 (s *0 2) *1 (1 *0 m *0 1) *1 (m *0 1) *1 m } \ar@3@/^5pc/[rrd] \ar@3[d]^{\raisebox{9 ex}{~~~~~$\twocell{s *0 2} \circ \doublebox{\twocell{(s *0 2) *1 (1 *0 m *0 1) *1 (m *0 1) *1 m}}$ }} &&
\\ 
\doublebox{\twocell{(m *0 2) *1 (s *0 1) *1 (m *0 1) *1 m }}\circ \twocell{m}= &
\twocell{( m *0 2) *1 (m *0 1) *1 m} \ar@3[dd]&
\twocell{(m *0 2) *1 (s *0 1) *1 (1 *0 m) *1 m } \ar@3[r] \ar@3[d]_{\twocell{(m *0 2) *1 g}~~~} &
\twocell{(1 *0 s *0 1) *1 (s *0 2) *1 (1 *0 m *0 1) *1 (1 *0 m) *1 m}  \ar@3[r] &
\twocell{(1 *0 s *0 1) *1 (s *0 m) *1 (1 *0 m) *1 m} \ar@3[d]  &
\twocell{(1 *0 s *0 1) *1 (1 *0 m *0 1) *1 (1 *0 m) *1 m} \ar@3[dl]  \ar@3@/^4pc/[ddl]_{\twocell{(1*0 g) *1 m}~~~~~} 
\\ &&
\twocell{(m *0 m ) *1 m} \ar@3[r] &
\twocell{(2 *0 m) *1 (1 *0 m) *1 m}  &
\twocell{(1 *0 s *0 1) *1 (2 *0 m) *1 (1 *0 m) *1 m} \ar@3[l] 
\\&
\twocell{( 1 *0 m *0 1) *1 (m *0 1) *1 m} \ar@3[ur]\ar@3[r] &
\twocell{(1 *0 m *0 1) *1 (1 *0 m) *1 m} \ar@3[rr]^{\raisebox{5 ex}{\twocell{penta}~~~~~~}} &&
\twocell{( 1 *0 m *0 1) *1 (1 *0 m) *1 m} \ar@3[ul]
}} $$
}

\item \twocell{(s *0 2) *1 (1 *0 m *0 1) *1 (s *0 1) *1 (1 *0 m) *1 m} ~: 

\resizebox{.5\width}{!}{
$$ \scalebox{0.9}{\xymatrix{
&&&
\twocell{(s *0 2) *1 (1 *0 m *0 1) *1 (m *0 1 ) *1 m } \ar@3@/^7pc/[rrddd] \ar@3[d]_{\twocell{(s*0 2)*1 (1 *0 m *0 1) *1 g}~~~~~}&&
\\
\Bigg (\doublebox{\twocell{(s *0 1) *1 (1 *0 m) *1 (s) *1 m }}* \twocell{1} \Bigg ) \circ \twocell{m}= &
\twocell{(s *0 2) *1 (1 *0 m *0 1) *1 (s *0 1) *1 (m *0 1 ) *1 m } \ar@3@/_2pc/[rdd]^{~~~=} \ar@3[rd] \ar@3@/^2pc/[urr]&
\twocell{(s *0 2) *1 (1 *0 m *0 1) *1 (s *0 1) *1 (1 *0 m) *1 m}\ar@3[d]\ar@3[r]&
\twocell{(s *0 2) *1 (1 *0 m *0 1) *1 (1 *0 m) *1 m} \ar@3[r]& 
\twocell{(s *0 m) *1 (1 *0 m) *1 m}  \ar@3[d]^{~~~~~~\raisebox{9 ex}{\doublebox{ \twocell{(s *0 2) *1 (1 *0 m *0 1) *1 (m *0 1) *1 m}}}} &
\\&
& \twocell{(1 *0 s *0 1) *1 (m *0 m) *1 m}   \ar@3[r] & 
\twocell{ (1 *0 s *0 1 ) *1 (1 *0 m *0 1) *1 (1 *0 m) *1 m}\ar@3[r]&
\twocell{(2 *0 m) *1 (1 *0 m) *1 m} &
\twocell{ (1 *0 m *0 1) *1 (m *0 1 ) *1 m } \ar@3[l]
\\ &
&
\twocell{(1 *0 s *0 1) *1 (1 *0 m *0 1) *1 (m *0 1 ) *1 m } \ar@3[r] \ar@3[ru]  \ar@3@/_2pc/[rrr]^{\raisebox{1ex}{=}}&
\twocell{(1 *0 s *0 1) *1 (1 *0 m *0 1) *1 (1 *0 m) *1 m } \ar@3[r] \ar@3[ru]^{=~~~~}&
\twocell{ (1 *0 m *0 1) *1 (1 *0 m) *1 m } \ar@3[ru]^{\twocell{(1 *0 g) *1 m}} &
\twocell{ (1 *0 m *0 1) *1 (m *0 1 ) *1 m } \ar@3[u]
}}$$
}

\item \twocell{(s *0 1) *1 (1*0 s) *1 (s *0 1) *1 (1 *0 m) *1 m}~: 

\resizebox{.5\width}{!}{
$$ \scalebox{0.9}{\xymatrix{
&
& \twocell{(s *0 1) *1 (m *0 1) *1 s *1 m} \ar@3@/_3pc/[dddl] \ar@3[dr] \ar@3[rrr] &&&
\twocell{(s *0 1) *1 (m *0 1)  *1 m} \ar@3[d] \ar@3@/^4pc/_{\twocell{g}~~}[ddd]
\\
\twocell{s *0 1} \circ \doublebox{\twocell{(m *0 1) *1 s *1 m}} = &
& \comm & \twocell{(m *0 1) *1 s *1 m} \ar@3^{\raisebox{6 ex}{=}}_{\raisebox{7 ex}{~}\doublebox{\twocell{(m *0 1) *1 s *1 m}}}[rr] \ar@3[d]&&
\twocell{ (m *0 1) *1 m} \ar@3[d]
\\&
& \twocell{(1 *0 s) *1 (s *0 1) *1 (1*0 s) *1 (1 *0 m) *1 m} \ar@3[r] &
\twocell{(1 *0 s) *1 (s *0 1) *1 (1 *0 m) *1 m} \ar@3[r]&
\twocell{(1 *0 s) *1 (1 *0 m) *1 m} \ar@3[r]&
\twocell{(1 *0 m) *1 m}
\\&
 \twocell{(s *0 1) *1 (1*0 s) *1 (s *0 1) *1 (1 *0 m) *1 m}  \ar@3[rr] \ar@3[ur]&&
\twocell{(s *0 1) *1 (1 *0 s) *1 (1 *0 m) *1 m} \ar@3[rr]&&
\twocell{(s *0 1) *1 (1 *0 m) *1 m} \ar@3[u]
}}$$
}

\columnbreak

\item \twocell{(s *0 2) *1 (1*0 s *0 1) *1 (s *0 m) *1 (1 *0 m) *1 m}~: 
if  \resizebox{.8\width}{!}{\doublebox{$\gamma$ vs $\alpha$}= \doublebox{ \twocell{(s *0 2) *1 (1 *0 m *0 1) *1 (m *0 1) *1 m}}}  and

 \resizebox{.8\width}{!}{ \doublebox{S} = \twocell{(s *0 2) *1 (1 *0 s *0 1)}$\circ $ \doublebox{ \twocell{(s *0 2) *1 (1 *0 m *0 1) *1 (m *0 1) *1 m}}}\; , then

\resizebox{.5\width}{!}{
$$ \scalebox{0.9}{
\xymatrix{
&
\twocell{(s *0 2) *1 (1*0 s*0 1) *1 (1*0 m *0 1) *1 (m *0 1) *1 m} \ar@3[dr]^{~~~~~~=} \ar@3[r] &
\twocell{(1*0 s*0 1) *1 (1*0 m *0 1) *1 (m *0 1) *1 m} \ar@3[dr] \ar@3[rr]&&
\twocell{(1*0 m *0 1) *1 (m *0 1) *1 m}\ar@3[d]_{\doublebox{$\gamma$ vs $\alpha$}~~~~~~} 
\\
\Bigg (\; \doublebox{\twocell{(s *0 1) *1 (1 *0 s) *1 (s *0 1) *1 (1 *0 m) *1 m}} * \twocell{1}\Bigg )\circ \twocell{m} =&
\twocell{(s *0 2) *1 (1*0 s*0 1) *1 (s *0 2) *1 (1*0 m *0 1) *1 (m *0 1) *1 m} \ar@3[d]^{~~~~~~ \doublebox{S}} \ar@3[u]\ar@3@/_3pc/[ddd]^{~~~=}&
\twocell{(s *0 2) *1 (1 *0 s*0 1) *1  (1*0 m *0 1) *1 (1 *0 m) *1 m} \ar@3[r]\ar@3[d]^{~~~~~~~=} &
\twocell{(1*0 s*0 1) *1 (1*0 m *0 1) *1 (1 *0 m) *1 m} \ar@3[d] &
\twocell{(1*0 m *0 1) *1 (1 *0 m) *1 m} \ar@3[d]
\\
&
\twocell{(s *0 2) *1 (1*0 s*0 1) *1 (s *0 2) *1 (1*0 m *0 1) *1 (1 *0 m) *1 m} \ar@3[dr]\ar@3[d]&
\twocell{(s *0 2) *1 (1*0 s*0 1) *1 (2 *0 m) *1 (1 *0 m) *1 m} \ar@3[r]& 
\twocell{(s *0 2) *1 (2 *0 m) *1 (1 *0 m) *1 m}   \ar@3[r] &
\twocell{(2 *0 m) *1 (1 *0 m) *1 m}
\\
&
\twocell{(1*0 s*0 1) *1 (s *0 2) *1 (1 *0 s *0 1) *1 (1*0 m *0 1) *1 (1 *0 m) *1 m} \ar@3[dr]^{\raisebox{6ex}{=}~~~}&
\twocell{(s *0 2) *1 (1*0 s*0 1) *1 (s *0 m) *1 (1 *0 m) *1 m}\ar@3[d]\ar@3[u]&
\twocell{(1 *0 s *0 1) *1 ( 2 *0 m )*1 (1 *0 m) *1 m}\ar@3[ru] &
\twocell{ (1 *0 g) *1 m } 
\\
& 
\twocell{(1*0 s*0 1) *1 (s *0 2) *1 (1 *0 s *0 1) *1 (1*0 m *0 1) *1 (m *0 1) *1 m} \ar@3[u] \ar@3[d]&
\twocell{(1 *0 s *0 1) *1 (s *0 2) *1 (1*0 s*0 1) *1 ( 2 *0 m )*1 (1 *0 m) *1 m} \ar@3[r] &
\twocell{(1 *0 s *0 1) *1 (s *0 2) *1 ( 2 *0 m )*1 (1 *0 m) *1 m} \ar@3[u]&
\twocell{(1 *0 s *0 1) *1 (1 *0 m *0 1) *1 (1 *0 m) *1 m} \ar@3[ul] \ar@3@/_2pc/[uu]
\\
&
\twocell{(1*0 s*0 1) *1 (s *0 2) *1 (1*0 m *0 1) *1 (m *0 1) *1 m} \ar@3[r] \ar@3[ru]^{~~~~~~~~\raisebox{6ex}{=}}_{~~~~~~~=}&
\twocell{(1*0 s*0 1) *1 (s *0 2) *1 (1*0 m *0 1) *1 (1 *0 m) *1 m} \ar@3@/^1pc/[ur] \ar@3[rr]^{~~~\raisebox{5 ex}{\twocell{1 *0 s *0 1} $\circ $ \doublebox{$\gamma$ vs $\alpha$} }} &&
\twocell{(1 *0 s *0 1) *1 (1 *0 m *0 1) *1 (m *0 1) *1 m} \ar@3[u] \ar@3@/_4pc/[uuuu]^{=~~}
}}$$
}

\item \twocell{(s *0 1) *1 (1 *0 s) *1 (m *0 1) *1 m}~:

\resizebox{.5\width}{!}{
$$ \scalebox{0.9}{
\xymatrix{
&
\twocell{(s *0 1) *1 (1 *0 s) *1 (m*0 1) *1 s *1 m} \ar@3[r] \ar@3^{~~~~~\comm}[d] \ar@3@/_3pc/[dd] &
\twocell{(s *0 1) *1 (1 *0 s) *1 ( m *0 1) *1 m} \ar@3[rr] \ar@3[d]   && 
\twocell{(s *0 1) *1 (1 *0 s) *1 (1 *0 m)*1 m} \ar@3[d]
\\ \twocell{(s *0 1) *1 (1 *0 s)} \circ \doublebox{\twocell{(m *0 1) *1 s *1 m}}= 
&
\twocell{(1 *0 m) *1 s *1 s *1 m} \ar@3[r] \ar@3@/_3pc/^{\raisebox{2 ex}{\twocell{(1 *0 m) *1 inv}}}[rr]   &
\twocell{(1 *0 m) *1 s *1 m} \ar@3[r] &
\twocell{(1 *0 m) *1 m} &
\twocell{(s *0 1) *1 (1 *0 m) *1 m} \ar@3[l]
\\
&
\twocell{(s *0 1) *1 (1 *0 s) *1 (1 *0 s) *1 (s *0 1) *1  (1 *0 m) *1 m} \ar@3@/_2pc/[rrr] \ar@3^{\raisebox{9 ex}{$\comm$~~~~~~~~~~~~~~~~~~~~~~~~~}}[r] &
\twocell{(s *0 1) *1 (s *0 1) *1 (1 *0 m) *1 m} \ar@3@/_2pc/[ur]&
\comm &
\twocell{(s *0 1) *1 (1 *0 s) *1 (1*0 s) *1 (1 *0 m)*1 m} \ar@3@/_4pc/^{\twocell{(s *0 1) *1 (1 *0 inv) *1 m}~}[uu]  \ar@3[u]
\\
}} $$
}
\columnbreak

\item \twocell{(s *0 2) *1 (1 *0 s *0 1) *1 (m *0 m) *1 m}~:

\resizebox{.5\width}{!}{
$$\scalebox{0.9}{\xymatrix{
&&
\twocell{(s *0 2) *1 (1 *0 s *0 1) *1 (1 *0 m *0 1) *1 (m *0 1) *1 m} \ar@3[r]\ar@3@/^7pc/[rrrd]_{=\raisebox{4ex}{~}}&
\twocell{(s *0 2) *1 (1 *0 s *0 1) *1 (1 *0 m *0 1) *1 (1 *0 m) *1 m}  \ar@3[d]_{\twocell{(s *0 2) *1 (1 *0 s *0 1) *1 penta}~~~~~~}^{~~~~~\twocell{(s *0 2) *1 (1 *0 g) *1 m}} \ar@3[r]&
\twocell{(s *0 2) *1 (1 *0 m *0 1) *1 (1 *0 m) *1 m} \ar@3[d]
\\
\Bigg ( \doublebox{\twocell{(s *0 1) *1 (1 *0 s) *1 (m *0 1) *1 m}} *\twocell{1}\Bigg ) \circ \twocell{m}= &
\twocell{(s *0 2) *1 (1 *0 s *0 1) *1 (m *0 2) *1 (m *0 1) *1 m}  \ar@3@/^2pc/[ur] \ar@3[r] \ar@3@/_2pc/[rdd]^{~~~~=}&
\twocell{(s *0 2) *1 (1 *0 s *0 1) *1 (m *0 m) *1 m} \ar@3[r] \ar@3[d]& 
\twocell{(s *0 2) *1 (1 *0 s *0 1) *1 (2 *0 m) *1 (1 *0 m) *1 m} \ar@3[r]&  
\twocell{(s *0 m) *1 (1 *0 m) *1 m} \ar@3[d] &
\twocell{(s *0 2) *1 (1*0 m *0 1) *1 (m *0 1) *1 m} \ar@3@/^6pc/[ddll]_{ \doublebox{ \twocell{(s *0 2) *1 (1 *0 m *0 1) *1 (m *0 1) *1 m}}} \ar@3[ul]
\\ 
&&
\twocell{(1*0 m *0 1) *1 (s *0 1) *1 (1 *0 m) *1 m} \ar@3[r]& 
\twocell{(1 *0 m *0 1) *1 (1 *0 m) *1 m} \ar@3[r] & 
\twocell{(2 *0 m) *1 (1 *0 m) *1 m}
\\ &&
\twocell{(1*0 m *0 1) *1 (s *0 1)*1 (m *0 1) *1 m}  \ar@3[u]_{~~~\twocell{(1 *0 m *0 1) *1 g}}\ar@3[r]&
\twocell{(1*0 m *0 1) *1 (m *0 1) *1 m} \ar@3[u]
}} $$
}

\end{multicols}
\end{itemize}

\end{document}